\documentclass{amsart}
\usepackage{fullpage}
\usepackage{include}

\newtheorem{algo}[thm]{Algorithm}

\begin{document}

\title{Numerical computation of Petersson inner products and $q$-expansions}
\author{Dan J. Collins}
\address{Mathematics Department \\ 
	University of British Columbia \\ 
	Vancouver, BC} 
\email{dcollins@math.ubc.ca} 


\begin{abstract}
	In this paper we discuss the problem of numerically computing Petersson
	inner products of modular forms, given their $q$-expansion at $\oo$. A
	formula of Nelson \cite{Nelson2015} reduces this to obtaining $q$-expansions
	at all cusps, and we describe two algorithms based on linear interpolation
	for numerically obtaining such expansions. We apply our methods to
	numerically verify constants arising in an explicit version of Ichino's
	triple-product formula relating $\g{fg,h}$ to the central value of $L(f\x
	g\x \bar{h},s)$, for three modular forms $f,g,h$ of compatible weights and
	characters.
\end{abstract}

\maketitle \tableofcontents

\section{Introduction}

The Petersson inner product on the space of holomorphic cusp forms $S_k(N,\ch)$
of a given weight, level, and character is a standard part of the theory of
modular forms, defined by (up to a normalizing factor)
\[ \g{f,g} = \S_{\DH\q\Ga} f(x+iy) \L{g(x+iy)} y^k \f{dx\ dy}{y^2}. \]
Specific values of this (and related integrals) arise often in the arithmetic
theory of newforms, their corresponding automorphic representations, and
associated geometric objects such as elliptic curves; in particular special
values of $L$-functions are often realized as such integrals. Thus it is of
interest to numerically compute such quantities. 
\par
We discuss how to compute $\g{f,g}$ given just the $q$-expansions of these
forms at $\oo$, and give some example applications of our method. Actually, the
problem we really consider is that of finding $q$-expansions of $f$ and $g$ at
all cusps, at which point we use a formula of Nelson \cite{Nelson2015} which
gives the Petersson inner product as a sum over all cusps $s$: 
\[ \g{f,g} = \f{4}{\vol(\DH\q\Ga)} 
		\s_s \f{h_{s,0}}{h_s} \s_{n=1}^\oo \f{a_{n,s} \Lb_{n,s}}{n^{k-1}} 
			\s_{m=1}^\oo \pf{x}{8\pi}^{k-1} \7( x K_{k-1}(x) - K_{k-2}(x) \7) 
				\qq\qq x = 4\pi m\rt{\f{n}{h_s}} \]
(this formula explained in more detail in Theorem \ref{thm:PeterssonFormula}).
\par
The computation of $q$-expansions of modular forms at cusps other than $\oo$
(given the $q$-expansion at infinity) is a surprisingly subtle problem, and the
main result of this paper is to give an algorithm that can numerically compute
these $q$-expansions for use in Nelson's formula. Recalling that the
$q$-expansion of $f$ at any cusp can be viewed as the $q$-expansion of
$f|[\al]_k$ at $\oo$ for some matrix $\al$, our approach is to calculate
various values of $f|[\al]_k$ (using the original $q$-expansion of $f$), and
then linearly interpolate these in a way that gives us a good numerical
approximation of the expansion at $\oo$. One version of our algorithm (assuming
absolutely nothing about $f$ beyond it being a modular form that we know the
$q$-expansion for) is Algorithm \ref{algo:LeastSquaresQExp}, which directly
interpolates the coefficients of the $q$-expansion. A second version is given
in Algorithm \ref{algo:LeastSquaresTwistsOfEigenform}, which assumes that $f$
is an eigenform away from bad primes and has the advantage that the computation
does not grow even as the number of coefficients we want does. 
\par
While we only discuss cusp forms and Petersson inner products in this paper, we
remark that this approach should be easily modified to other situations.
Nelson's formula can be applied to general integrals of automorphic functions
on quotients of the upper half-plane. Certainly any other sort of integral
constructed from modular forms could be handled this way, and our interpolation
approach could be modified to handle other classes of functions that can be
described reasonably in terms of a Fourier expansion (e.g. Maass forms). 

\subsection*{Our motivation, and comparison with other approaches.} Our
specific motivation for studying this comes from the situation where we have
\E{three} newforms $f,g,h$ such that the product $fg$ has the same weight and
character as $h$. A general formula of Ichino \cite{Ichino2008} gives a
relation between $|\g{fg,h}|^2$ and the central value of a triple-product
$L$-function which we may write as 
\[ |\g{fg,h}|^2 = C\dt L(f\x g\x\Lh,m-1) \dt \p_{\tx{bad primes }p} I^{**}_p,\]
where the constant $C$ and the local constants at bad primes $I^{**}_p$ are
things that can be in principle evaluated from the setup of the problem, but in
practice the computations are quite subtle. In \cite{Collins2016} we establish
a completely explicit formula in some cases, and use it to construct $p$-adic
$L$-functions. 
\par
In a context like this it is important to know that the algebraic part of our
constants are precisely correct, because we ultimately want to study
$p$-integrality and congruences modulo $p$ for our $p$-adic $L$-function.
Hence, we wish to numerically compute the ratio of $|\g{fg,h}|^2$ and $L(f\x
g\x\Lh,m-1)$ in many cases and verify this agrees with the constants we obtain
in our formula. Numerical agreement in a representative sample of examples
provides a very convincing argument that the constants are indeed correct,
because errors in the theoretical calculations generally result in things like
the constants containing extraneous powers of $2$ or incorrect Euler-like
factors such as $(1+1/p)$.
\par
To implement this calculation, there is a well-known algorithm of Dokchitser
\cite{Dokchitser2004} that we can use to compute the $L$-value. However, we
were not able to find in the literature a satisfactory method for computing
Petersson inner products for our purposes. Ideally, we would like our algorithm
to have the following characteristics:  
\begin{itemize}
	\item Works directly with the $q$-expansions of our modular forms at
		infinity, since this is how our modular forms are given.
	\item Avoids computing with full spaces of cusp forms as much as possible; 
		in examples we want to test $f,g,h$ may all be of reasonably large levels
		that are coprime to each other, so any space $S_k(N,\ch)$ containing both
		$fg$ and $h$ may be of large enough dimension to make it impractical to
		work with.
\end{itemize}
The most commonly-suggested method, perhaps, is to us the connection with
adjoint $L$-functions - for a newform $f$, there is an explicit formula
relating between $\g{f,f}$ and $L(\ad f,1)$. However, using this for something
like $\g{fg,h}$ requires decomposing $fg$ in terms of an eigenbasis, which
ultimately would involve computing a full space of cusp forms that is
potentially very large. Also, we will see in Section
\ref{sec:ComparisonWithAdjointLValues} that it is a nontrivial task just to
implement the formula relating $\g{f,f}$ and $L(\ad f,1)$ for newforms of
arbitrary level! Another approach is given in \cite{Cohen2013}, but this is
based on numerical integration from the values of the function itself, which
isn't ideal for modular forms given as $q$-expansions.
\par
The most promising approach seemed to be to use Nelson's formula, which
expresses the Petersson inner product as a straightforward infinite sum
(involving some $K$-Bessel functions) over the $q$-expansions. Of course, this
requires a method to get the $q$-expansions at other cusps, and once again
there are an assortment of results in the literature but none that were
satisfactory for our purposes. Asai \cite{Asai1976} uses Atkin-Lehner operators
to give a full expression of expansions at all cusps for modular forms of
squarefree level, but there are not any results nearly as nice for the general
case. Some partial results are given in the thesis of Delaunay
\cite{Delaunay2002}, and a formula and algorithm for expansions at cusps of
width one was given in the recent thesis of Chen \cite{Chen2016}. The only
general algorithm we are aware of is in Section 3.6.8 of the book
\cite{EdixhovenCouveignes}, but this involves computations with a full space
of modular forms (actually, of even higher level than what one starts with)
so would be impractical for the applications we have in mind.

\subsection*{Overview of this paper.} In Section
\ref{sec:NumericalExpansionAtCusps} we present the core results of this paper:
setting up the problem of determining $q$-expansions at all cusps, and then
presenting our algorithms for numerically computing these expansions. Section 
\ref{sec:LeastSquaresQExpansion} presents our first algorithm, which solves for
the coefficients of $f|[\al]_k = \s b_n q^n$ by truncation of the sum and
direct interpolation of the coefficients $b_n$. Our second algorithm, in
Section \ref{sec:LeastSquaresEigenbasis}, applies to the case that $f$ is an
eigenform and instead interpolates $f|[\al]_k$ as a linear combination of a
basis for the eigenspaces of $f$ and its twists. The theoretical result
guaranteeing that $f|[\al]_k$ arises as such a linear combination is the
following: 

\begin{thm} 
	Let $f \in S_k(N,\ch)$ be an eigenform of the Hecke operators $T_p$ for
	$p\nd N$ (i.e. an oldform associated to a newform $f_0 \in S_k(N_0,\ch)$ for
	some $N_0|N$). Then $f|[\al]_k$ (its expansion at another cusp, normalized
	to have integer exponents in its $q$-expansion) is a linear combination of
	twists $(f_0\ox\mu)(mz)$ that lie in $S_k(\Ga_1(Nh))$. 
\end{thm}

This is stated later on as Theorem \ref{thm:LinearCombinationOfTwists}, which
is proven in Section \ref{sec:TransformEigenformToOtherCusp}. In Section 
\ref{sec:NarrowingDownEigenspace} we discuss how to narrow down the space
$S_k(\Ga_1(Nh))$ in which $f|[\al]_k$ may live, and thus the list of twists
potentially needed. We remark that determining all of the twists of the
appropriate level requires knowing the minimal-level twist of $f_0$. Finding
this minimal level twist is the only place our current algorithm may require
working with a full space of cusp forms $S_k(N,\ch)$; we discuss this and
potential ways to avoid it in Section \ref{sec:ComputingMinimalTwists}.
\par
We combine our $q$-expansion algorithms with Nelson's formula in Section
\ref{sec:ComputingPeterssonInnerProducts} to describe an algorithm for
numerically computing Petersson inner products. This is followed with some
examples of computing self-Petersson inner products $\g{f,f}$ for newforms $f$,
and comparing with the known formula for $\g{f,f}$ in terms of $L(\ad f,1)$,
plus some computations of ratios of Petersson inner products such as
$\g{f(pz),f(z)}/\g{f(z),f(z)}$ which are relevant in the study of $p$-adic
$L$-functions. In Section \ref{sec:InnerProductsThreeEigenforms} we describe
how to best implement our methods to compute products $\g{fg,h}$, and then
describe several computations we have made to verify formulas proven in 
\cite{Collins2016}. 

\subsection*{Acknowledgements.} The author would like to thank Peter Humphries,
Paul Nelson, Nicolas Templier, David Zywina, and Vinayak Vatsal for helpful
conversations about how to approach this problem throughout the course of this
project.
 
\section{Approaches to numerical computation of $q$-expansions at cusps} 
\label{sec:NumericalExpansionAtCusps}
\subsection{Precise setup of the problem} 

Before describing our methods for computing the $q$-expansion of a modular form
at all cusps, we want to be precise about how we're formulating the problem and
about what spaces all of the relevant modular forms live in. Throughout we will
let $f \in M_k(N,\ch)$ be a modular form of weight $k$ on $\Ga_0(N)$ with 
character $\ch$. Our goal is to start with the $q$-expansion 
\[ f(z) = \s a_n e^{2\pi i nz} = \s a_n q^n \]
of $f$ at infinity and, from that, compute the $q$-expansions of the
translates 
\[ f|[\al]_k(z) = (cz+d)^{-k} f\pf{az+b}{cz+d} 
		\qq\qq \al = \M{cc}{ a & b \\ c & d } \] 
for all choices of $\al \in \SL_2(\Z)$. Of course since we know how $f$
transforms under $\Ga_0(N)$ this reduces to looking at finitely many matrices
representing the cosets of $\Ga_0(N)\q\SL_2(\Z)$. 
\par
The problem can be further condensed by passing from a matrix $\al$ as above to
the corresponding \E{cusp} in $\P^1(\Q)$, which we take to be the image of
$\oo$ under the action of $\al$ by a M\"obius transformation: $\al\oo = a/c$.
If two matrices $\al,\be$ correspond to the same cusp, we will explicitly
describe how the $q$-expansions differ at the end of this section. So we really
just need to understand $f|[\al]_k$ for one matrix $\al$ corresponding to each
cusp. An explicit description of the cusps can be given as in Proposition 1.43
of \cite{ShimuraAutForms}; all we'll really need is that each non-$\oo$ cusp
can be represented as $a/c$ for $c$ a proper divisor of $N$ and $(a,c) = 1$.
\par
So now we consider a cusp $a/c$ of this form, and fix a choice of matrix 
\[ \al_1 = \M{cc}{ a & b \\ c & d } \in \SL_2(\Z). \]
We know $f|[\al_1]_k$ is a modular form for the group $\al_1\1\Ga_0(N)\al_1$
with character induced by $\ch$ under conjugation, which is a congruence
subgroup containing $\Ga(N)$. However, it does not contain $\Ga_1(N)$ and thus
the $q$-expansion of $f|[\al_1]_k$ may involve fractional powers. To avoid this
we replace $f|[\al_1]_k(z)$ by some $f|[\al_1]_k(hz)$ which is a modular form
in some $M_k(\Ga_1(N'))$, ideally with $h$ as small as possible. We can
equivalently write $f|[\al_1]_k(hz)$ as (a scalar multiple of) $f|[\al_h]_k$
for 
\[ \al_h = \al_1 \dt \ta_h = \M{cc}{ a & b \\ c & d } \M{cc}{ h & 0 \\ 0 & 1 } 
		= \M{cc}{ ah & b \\ ch & d }. \]

\begin{lem} \label{lem:WidthOfCuspForModularForm}
	Fix $N$, $\ch$, and $a/c$ as above. Let $h|(N/c)$ be an integer satisfying
	both 
	\begin{itemize}
		\item $N$ divides $c^2 h$.
		\item $\ch$ is trivial on the subgroup $(1 + ch\Z)/N\Z$ of $(\Z/N\Z)^\x$. 
	\end{itemize}
	Then for any $f \in M_k(N,\ch)$, we have $f|[\al_h]_k \in M_k(\Ga_1(Nh))$. 
\end{lem}
Note that the smallest $h$ satisfying the first condition is exactly the width
of the cusp $a/c$ for $\Ga_0(N)$, and the smallest $h$ satisfying both is the
width of $a/c$ for $\ker\ch \le \Ga_0(N)$. So this $h$ is indeed the smallest
integer such that $[\al_h]_k$ takes $M_k(N,\ch)$ into any $M_k(\Ga_1(N'))$. 
\begin{proof}
	The first step is showing that $\Ga_1(Nh)$ is a subgroup of the group
	$\al_h\1\Ga_0(N)\al_h$ for which $f|[\al_h]_k$ is modular; equivalently, we
	have to show that if $\ga \in \Ga_1(Nh)$ then $\al_h \ga \al_h\1 \in
	\Ga_0(N)$. If we write 
	\[ \ga = \M{cc}{ A & B \\ C & D } \ee \M{cc}{ 1 & * \\ 0 & 1} \md{Nh}, \]
	then an explicit calculation (using $c^2 h \ee 0\md N$) shows that 
	\[ \al_h\ga\al_h\1 \ee \M{cc}{ 1 - Bach & ** \\ 0 & 1 + Bach } \md N. \]
	This means $f|[\al_h]_k$ lies in $M_k(\Ga_1(N),\ch')$ for $\ch'$ the
	character given by $\ch'(\ga) = \ch(\al_h\ga\al_h\1)$. Since we don't want a
	character on $\Ga_1(N)$ we need to insist that this is trivial, i.e. that
	we've chosen $h$ large enough so the elements $1\pm Bach$ on the diagonal
	are actually in the kernel of $\ch$. 
\end{proof}

The main goal of this paper is to present practical methods for determining the
$q$-expansion $f|[\al_h] = \s b_n q^n$ for any cusp $a/c$, working from the
original $q$-expansion $f = \s a_n q^n$. In some cases there is a satisfactory
theoretical way to find $f|[\al_h]$ using Atkin-Lehner operators (we will
discuss this, as well as a more general refinement of the above Lemma, in
Section \ref{sec:NarrowingDownEigenspace}). But if the level $N$ is divisible
by large powers of a prime, then the exact determination of $f|[\al_h]$ is a
delicate problem in local representation theory. So instead we will look for a
way to \E{numerically} compute the coefficients $b_n$.

\subsubsection*{Expansions for other matrices at the same cusp.} When expanding
at a cusp $a/c$ we'll usually work with a fixed matrix $\al_1$ as above, but in
some cases we'll need to consider other matrices too. Suppose
\[ \be_1 = \M{cc}{ a' & b' \\ c' & d' } \in \SL_2(\Z) \]
is any other matrix that takes $\oo$ to the cusp $a/c$ of $\Ga_0(N)$. Cusps can
be described as double cosets in $\Ga_0(N)\q \SL_2(\Z)/\Ga_\oo$ where $\Ga_\oo$
is the stabilizer of the cusp infinity in $\SL_2(\Z)$, i.e. $\Ga_\oo = \{ \pm
\de_x : x\in\Z \}$ where we write 
\[ \de_x = \M{cc}{1 & x \\ 0 & 1}. \]
So, if $\al_1$ and $\be_1$ represent the same cusp, there is $\ga \in \Ga_0(N)$
and $x\in\Z$ with $\be_1 = \ga\al_1(\pm\de_x)$. If we set $\be_h = \be_1\ta_h$
we then get 
\[ f|[\be_h] = \ch(\ga) \6( f|[\al_h]_k \6) | [\ta_h\1\de_x\ta_h]_k. \]
A computation gives that $\ta_h\1\de_x\ta_h = \de_{x/h}$, so $f|[\be_h]_k$ is
equal to $f|[\al_h]_k$ with the slash operator $[\de_{x/h}]_k$ applied and times
a constant. It's straightforward to check that $\de_{x/h}$ normalizes
$\Ga_1(Nh)$ so $f|[\be_h]_k$ still lies in $M_k(\Ga_1(Nh))$. Also, $\de_{x/h}$
acts in a predictable way on the $q$-expansion, which we summarize in the
following proposition.

\begin{prop} \label{prop:ExpansionAtEquivalentCusps}
	Suppose 
	\[ \be_1 = \M{cc}{ a' & b' \\ c' & d'} \qq\qq 
			\be'_1 = \M{cc}{ a'' & b'' \\ c'' & d''} \]
	are two matrices taking $\oo$ to the same cusp $a/c$ for $\Ga_0(N)$, with
	width $h$ as in the above proposition. If $f|[\be_h]_k$ has $q$-expansion
	$\s b_n q^n$, then we have 
	\[ f|[\be'_h]_k = \ch\6(a'd''h - b' c'' - a' c'' x\6) 
			\s b_n \exp(2\pi i nx/h) q^n \] 
	where $x$ is an integer chosen such that $c''d'-c'd''h+c'c''x \ee 0 \md N$.
\end{prop}
\begin{proof}
	The claim that $\be_1,\be'_1$ take $\oo$ to the same cusp $a/c$ means that
	they are in the same double coset in $\Ga_0(N)\q \SL_2(Z)/\Ga_\oo$, i.e.
	that there's $\ga \in \Ga_0(N)$ and $\de_x \in \Ga_\oo$ (where we WLOG move
	the factor of $\pm I$ to the matrix in $\Ga_0(N)$) such that $\be'_1 =
	\ga\be_1 \de_x$. Right-multiplying by $\ta_h$ and rearranging we get 
	\[ \ga = \be'_h \de\1_{x/h} \be_h\1. \]
	Computing out the product on the right-hand side we find the bottom-left
	entry is $c'd''-c''d+c'c''x$, so our assumption that $\ga \in \Ga_0(N)$
	forces $x$ to satisfy the specified congruence. Since $\be'_h = \ga \be_h
	\de_{x/h}$, we can compute the $q$-series of $f|[\be'_h]$ by applying these
	three matrices - $\ga$ transforms $f$ via $\ch$ applied to its lower-right
	entry, $\be_h$ gives the expansion $\s b_n q^n$, and $\de_{x/h}$ replaces
	$q$ by 
	\[ \exp(2\pi i n(z+x/h)) = q \dt \exp(2\pi i nx/h). \QED \]
\end{proof}

\subsubsection*{Expansions of $f(mz)$ in terms of expansions of $f(z)$.} If one
has a modular form $f(z)$ and applies a degeneracy map to it to obtain a
modular form $f(mz)$ for some positive integer $m$, the expansion of $f(mz)$
at any cusp can be obtained from the expansions of $f(z)$ at a possibly
different cusp. We will describe explicitly how to do this here; note that this
reduces the problem of finding expansions of eigenforms just to the case of
newforms. 
\par
It is helpful to consider the case where $m = p$ is as prime, which divides up
into two situations: the case where $p \nd c$ (the denominator of our cusp) and
the case $p|c$. In the former case we can choose our matrix $\al_1$ to have
$d|p$, at which point we write 
\[ f(z) \9|_k \M{cc}{p & 0 \\ 0 & 1}\9|_k \M{cc}{ a & b \\ c & d } 
	= f(z) \9|_k \M{cc}{ ap & b \\ c & d/p } \9|_k \M{cc}{1 & 0 \\ 0 & p} \]
In the latter case we instead have 
\[ f(z) \9|_k \M{cc}{p & 0 \\ 0 & 1}\9|_k \M{cc}{ a & b \\ c & d } 
	= f(z) \9|_k \M{cc}{ a & bp \\ c/p & d } \9|_k \M{cc}{p & 0 \\ 0 & 1}. \]
If $m$ is composite we can iterate this procedure one prime at a time to get
that $f(mz)|[\al]_k$ is equal to $(f|[\al'])(m'z)$ for some matrix $\al' \in
\SL_2(\Z)$ and some rational number $m'$.
\par
To give the general case explicitly, suppose $f \in M_k(N,\ch)$ is a modular
form, $m$ is an integer, and we want to consider the expansion of $f(mz) \in
M_k(Nm,\ch)$ at a cusp $a/c$ of $\Ga_0(Nm)$. As usual we assume $c|N$ and
$(a,c) = 1$, and fix a matrix 
\[ \al_1 = \M{cc}{ a & b \\ c & d } \in \SL_2(\Z) \]
taking $\oo$ to $a/c$. Let $m_1 = (c,m)$ and $m_2 = m/m_1$; note that this
implies $c/m_1$ and $m_2$ are coprime, and therefore we may find an integer $y$
such that $d - (c/m_1)y$ is divisible by $m_2$. Then we have 
\[ f(mz)|[\al_1]_k = m^{-k/2} \dt f(z) 
		\9|_k \M{cc}{m & 0 \\ 0 & 1} \9|_k \M{cc}{ a & b \\ c & d } 
	=  m^{-k/2} \dt f(z) 
		\9|_k \M{cc}{ am_2 & bm \\ c/m_1 & d } \9|_k \M{cc}{ m_1 & 0 \\ 0 & 1 }\]
and we can further expand 
\[ \M{cc}{ am_2 & bm \\ c/m_1 & d } 
	= \M{cc}{am_2 & bm - yam_2 \\ c/m_1 & d - yc/m_1 } \M{cc}{ 1 & y \\ 0 & 1 }
	= \M{cc}{am_2 & bm_1 - ya \\ c/m_1 & \f{d - yc/m_1}{m_2} } 
				\M{cc}{ 1 & 0 \\ 0 & m_2 } \M{cc}{ 1 & y \\ 0 & 1 }, \]
at which point our initial expression is written in terms of an expansion of
$f$ at the cusp $(am_2)/(c/m_1)$. 


\subsection{Attempt via Fourier analysis} 

As above, suppose we have $f = \s a_n q^n \in M_k(N,\ch)$ and that we want to
compute the coefficients in the expansion $f|[\al_h]_k = \s b_n q^n$ at a cusp
$a/c$. A first approach one might try is to simply use Fourier inversion to
obtain a formula for each $b_n$. We describe this computation, and why it does
not turn out to give us a practical algorithm. 
\par
We can single out the Fourier coefficient $b_m$ by integrating
$f|[\al_h]_k(x+iy) \dt \exp(-2\pi i m(x+iy))$ from $x = 0$ to $x = 1$ (for a
fixed value of $y$):
\begin{multline*}
	b_m = \S_0^1 \b( \s b_n q^n \e) \exp(-2\pi i m(x+iy)) dx 
		= \exp(2\pi my) \S_0^1 f|[\al_h]_k(x+iy) \exp(-2\pi i mx) dx \\
		= \exp(2\pi my) \S_0^1 h^{k/2} (ch(x+iy)+d)^{-k} 
				f\pf{ah(x+iy)+b}{ch(x+iy)+d} \exp(-2\pi imx) dx.
\end{multline*}
Since $f(z) = \s a_n q^n$ we can simply substitute this in and rearrange to get
\[ b_m = h^{k/2} \exp(2\pi my) \s_{n=0}^\oo a_n 
		\S_0^1 \f{1}{(chx + d + ichy)^k} \exp\b(2\pi i 
				\9( n\f{ahx + b + iahy}{chx + d + ichy} - mx \9) \e) dx. \]
This gives a series converging to $b_m$. However, it does not seem to be
practical to compute $b_m$ this way - the series can take quite a while to
converge, and without a very efficient method for computing the integrals (for
all values of both $m$ and $n$ up to whatever cutoffs we need) the computation
will be very slow. 


\subsection{Approach 1: Least squares for the $q$-series} 
\label{sec:LeastSquaresQExpansion}

Another approach to determining the Fourier coefficients of $f|[\al_h]_k = \s
b_n q^n$ is to treat the $b_n$'s as variables to be filled in by interpolating
from the known values that the function takes. As stated this has infinitely
many variables, but truncating we can approximate it as $\s_{n=0}^K b_n q^n$.
We can evaluate $f|[\al_h]_k(z)$ at many points, and try to find the
coefficients $b_0,\ld,b_K$ that best fit the data. 
\par
If we choose points $z_1,\ld,z_M$ on the upper half-plane, and let $q_j =
\exp(2\pi i z_j)$, then after computing each $q_j$ and its powers plus each
value $f|[\al_h](z_j)$ (from the original $q$-expansion of $f$), the problem is
to choose the vector of values $b_0,\ld,b_K$ that offers the best solution to
the matrix equation 
\[ \M{ccccc}{ 1 & q_1 & q_1^2 & \cd & q_1^K \\
					1 & q_2 & q_2^2 & \cd & q_2^K \\
					\vd & \vd & \vd & \dd & \vd \\
					1 & q_M & q_M^2 & \cd & q_M^K }
		\M{c}{ b_0 \\ b_1 \\ b_2 \\ \vd \\ b_K } 
		= \M{c}{ f|[\al_h](z_1) \\ f|[\al_h](z_2) \\ \vd \\ f|[\al_h](z_M) } \]
where $q_l = \exp(2\pi i z_l)$.
\par
If we interpret ``best solution'' as asking for the smallest Euclidean distance
between the two sides as elements of $\R^M$, then this is just a standard
problem in linear algebra, and the least-squares solution to the equation $Ax
= b$ is the actual solution to $(A^* A) x = A^* b$ where $A^*$ is
the conjugate transpose of $A$. It's then straightforward to implement this as
an algorithm: given $M$, $K$, and the points $z_1,\ld,z_M$ we can compute the
matrix of powers of $q$ and the vector of values of $f|[\al_h]$ as
floating-point complex numbers, and then perform solve the floating-point
linear system $(A^* A) x = A^*b$.
\par
The next question is how to best choose $M$, $K$, and the points $z_j$. The
number $K$ of coefficients to look for and the imaginary parts of the $z_j$ are
closely related to the accuracy we want from the calculation. Specifically,
since we've chosen as $K$ as our cutoff, then $\s_{j=0}^K b_j q^j$ will differ
from the actual value of $f|[\al_h](z)$ by the tail $\s_{j=K=1}^\oo b_j q^j$,
which is on the order of $|q^K| = \exp(-2\pi K\Im(z))$, so every part of our
computation will have an error of around this size. Also, when determining the
accuracy of the coefficient $b_j$, it's actually the product $b_j q^j$ which
can be expected to have error of size $\exp(-2\pi K\Im(z))$, so the error of
$b_j$ will be about the order of $\exp(-2\pi(K-j)\Im(z))$. 
\par
So, what we can do is specify a number $K_0$ of coefficients we definitely
want, an absolute error $10^{-E}$ for our calculations, and an exponential
decay rate $e^{-C_0}$ such that we'd like the error of $b_j e^{-C_0}$ to be on
the order of $10^{-E}$. (This is a reasonable requirement, because for our
applications we'll be computing sums where $b_j$ is multiplied by some
exponentially-decreasing factor). For the actual computation we need to aim for
an exponential decay rate $C$ and a number of coefficients $K$ such that
$e^{-KC} \ap 10^{-E}$ and thus we truncate our sum at around the correct place,
so start with $K = K_0$ and $C = C_0$ and either increase $K$ or decrease $C$
to get $KC \ap \log(10) E$. 
\par
To be able to compute the coefficients with decay rate $e^{-C}$, we sample at
points $z_j$ where $|q_j| \ap e^{-C}$, i.e. $\Im(z_l) \ap C/2\pi$. Moreover,
when computing the values of $f|[\al_h](z_j)$ the factor of automorphy
$(chz_j+d)$ affects location of the translated point $\al_h z_j$ and thus the
speed of convergence of the sum, so to optimize this we prefer to choose points
$z_j$ with $\Re(z_j) \ap -d/ch$ to minimize this. 
\par 
In our implementations, we chose points with $\Im(z_l) = C/2\pi$, and with
$\Re(z_l)$ chosen randomly in an interval of length 1 centered at $-d/ch$.
Fixing the imaginary part leaves the magnitude of all of our computations
equal. Since we're working directly with powers of $\exp(2\pi i z)$ that are
periodic under $z \mt z+1$ there's no reason to work outside of an interval of
length 1, but the interpolation seems somewhat sensitive to working in any
smaller range. The number of points sampled $M$ needs to be at least as large
as $K$ for our interpolation problem to be solvable in principle, and the
larger $M$ is the more accurate the computation is likely to be; we settled on
$M = 2K$ as a workable choice. 

\begin{algo}[Least-squares for $q$-expansion] \label{algo:LeastSquaresQExp}
	Suppose we have a modular form $f = \s a_n q^n \in M_k(N,\ch)$ and we want
	to compute its expansion $f|[\al_h] = \s b_n q^n$ at a cusp given by a
	matrix $\al_h$ in our notation above. Suppose further that we've fixed
	constants $E$, $K_0$, and $C_0$ such that for $n\le K_0$ we would like to
	compute the coefficient $b_n$ to with an error of approximately $10^{-E}
	e^{nC_0}$. We proceed as follows:
	\begin{itemize}
		\item Either increase $K = K_0$ or decrease $C = C_0$ so that $KC \ap
			\log(10) E$, and work with interpolating the truncation
			$\s_{n=0}^{K_0} b_n q^n$ of the expansion for $f|[\al_h]$. 
		\item Choose $M$ (we used $2K_0$) and pick $M$ points $z_1,\ld,z_M$ with
			$\Im(z_j) = C/2\pi$ and $\Re(z_l)$ is picked randomly in the interval
			of length 1 centered around $-d/ch$ (for $c,d,h$ the parameters from
			the matrix $\al_h$).
		\item Numerically compute the values $f|[\al_h](z_l) = h^{k/2} (chz_l +
			d)^{-k} f(\al_h z_j)$ using the $q$-expansion for $f$, truncating
			when we've reached an accuracy a bit past $10^{-E}$, and fill these
			into a vector $b$.
		\item Numerically compute the values $q_j^n = \exp(2\pi i n z_j)$ and
			fill these into a matrix $A$.
		\item Numerically find the least squares solution to $Ax = b$ as the
			exact solution to $(A^* A)x = A^* b$. The solution vector $x$ is our
			numerical approximation to the coefficients $b_0,b_1,\ld,b_K$.
	\end{itemize}
\end{algo}

Given the nature of the least-squares approximation, it seems very unlikely to
be able to establish rigorous error bounds for this algorithm (even if the
points were picked deterministically rather than randomly). Nonetheless it
seems to work well in practice, and testing with various examples it returns
values for the coefficients with accuracy close to what we hope. 
\par
For example, consider the unique newform 
\[ f = q - 2q^2 - 3q^3 + 4q^4 + 6q^5 + 6q^6 - 16q^7 - 8q^8 + \cd 
		\in S_4(\Ga_0(6)); \]
because this has squarefree level the results of Asai \cite{Asai1976} tell us
that its expansion at any cusp should be a multiple of itself. Sure enough, if
we run the algorithm above with $E = 15$ and $C = 1$, we need to compute $K =
35$ coefficients and thus sample at $70$ points. An example run of this for the
cusp $1/3$ and the matrix 
\[ \al_1 = \M{cc}{ 1 & -1 \\ 3 & -2 } \]
and $h = 2$ required using around $270$ coefficients of $f$ for the
slowest-converging sum, and returns that $f|[\al_2]_k$ is approximately 
\small
\begin{align*} 
	&\ss(1.0000000000000147 + .0000000000000235i)q 
	+(-1.9999999999999052 + .0000000000000885i)q^2 \\
	&\ss\qu+(-2.9999999999996767 - .0000000000002597i)q^3
	+(3.9999999999998517 + .0000000000000770i)q^4\\
	&\ss\qu+(5.9999999999967810 + .0000000000018893i)q^5
	+(6.0000000000018602 - .0000000000051318i)q^6 + \cd,
\end{align*}
\normalsize
which is an approximation of $f$ itself with errors on the scale we wanted. 
\par
Expansions at cusps for non-squarefree levels can get more complicated and seem
less well-understood theoretically. For instance, one can take the newform
\[ f = q - 3q^2 + q^4 - 15q^5 - 25q^7 + 21q^8 + 45q^{10} + \cd 
		\in S_4(\Ga_0(27)) \]
and looks at the cusp $1/3$ where we take the same matrix $\al_1$ as above but
this time with width $h = 3$. If we want $E =  15$ and $C = 1$ once again we
find we need to take $K = 35$ and sample at 70 points. This time a sample
run-through used approximately 410 coefficients of $f$ for its
slowest-converging sum, and returns that $f|[\al_3]_k$ is approximately 
\small
\begin{align*}
	&(.9396926207858713 - .3420201433255586i)q
	+(2.2981333293573119 - 1.9283628290595167i)q^2 \\
	&\qu+(-.0000000000004964 - .0000000000003253i)q^3
	+(-.1736481776683433 + .9848077530113447i)q^4 \\
	&\qu+(2.6047226650051819 +14.7721162951836733i)q^5
	+(-.0000000000019237 - .0000000000090777i)q^6 + \cd
\end{align*}
\normalsize
Here the coefficients are much less readily recognizable, but one can identify
the first coefficient as being the inverse of the usual primitive 18th root of
unity $\ze_{18}$. Similarly the other coefficients appear to also be related to
18th roots of unity times the corresponding coefficient of the original modular
form $f$, and our computations suggest 
\[ f|[\al_3]_k = \ze_{18}^{-1} q + 3\ze_{18}^{-2} q^2 + 0 q^3 + \ze_{18}^5 q^4 
		+ 15\ze_{18}^4 q^5 + 0 q^6 + \cd. \]
In the next section we will approach this problem from a different angle and
make it somewhat more clear where these coefficients are coming from.


\subsection{Approach 2: Least squares for an eigenbasis} 
\label{sec:LeastSquaresEigenbasis}

A downside to the least-squares algorithm applied to $q$-expansions is that if
we need many coefficients of our modular form (which will happen when we
compute Petersson inner products using Nelson's formula), the algorithm gets
quite slow: to obtain $M$ coefficients we need to compute values at $2M$ points
and then numerically solve a least-squares problem for a $2M\x M$ matrix. But
modular forms are determined by only a finite number of coefficients, so in
principle we should be able to make this computation independent of the number
of coefficients we want. 
\par
One way to accomplish this is to simply compute a basis of the space
$M_k(\Ga_1(Nh))$ containing $f|[\al_h]_k$, and then perform a least-squares
computation to find a best approximation of $f|[\al_h]_k$ as a linear
combination of this basis by evaluating at a collection of points in the upper
half-plane. If our basis consists of $d$ modular forms, then evaluating at
$2d$ points should give us a good numerical approximation of the coefficients
of the linear combination from which we can recover numerical approximations
for any number of coefficients we want. The downside of this naive approach is
that the dimension $d$ of $M_k(\Ga_1(Nh))$ grows linearly in terms of the
weight $k$ and quadratically in terms of the level $Nh$, and for even fairly
small levels and weights $d$ may end up much larger than the number of
coefficients we want to obtain. 
\par
So if $f$ is an arbitrary modular form in $M_k(N,\ch)$ then it seems unlikely
that a least-squares approach attempting to realize $f$ as a linear combination
of other modular forms would be efficient. However, for most of the examples we
care about $f$ is far from arbitrary: the modular forms $f$ of most interest
are eigenforms. In this case we could hope that $f|[\al_h]_k$ is a linear
combination of a comparatively small number of basis elements. Indeed this is
true; the following theorem will be proven in Section
\ref{sec:TransformEigenformToOtherCusp}. (We restrict to cuspidal eigenforms at
this point, because our interest is in modular forms in the old subspace
corresponding to a particular newform, but the argument should extend to
Eisenstein series as well). 

\begin{thm} \label{thm:LinearCombinationOfTwists}
	Let $f \in S_k(N,\ch)$ be an eigenform of the Hecke operators $T_p$ for
	$p\nd N$ (i.e. an oldform associated to a newform $f_0 \in S_k(N_0,\ch)$ for
	some $N_0|N$). Then $f|[\al_h]_k$ is a linear combination of twists
	$(f_0\ox\mu)(mz)$ that lie in $S_k(\Ga_1(Nh))$. 
\end{thm}

Here $f_0\ox\mu$ denotes the newform that is a twist of $f_0$ by a Dirichlet
character $\mu$, so $f_0\ox\mu$ may differ from the ``naive twist''
$f_{0,\mu} = \s \mu(n)a_n q^n$ which may not be a newform itself (but is an
oldform associated to the newform $f_0\ox\mu$). 
\par
This result gives us a reasonably small subspace of $S_k(\Ga_1(Nh))$ to look
for $f|[\al_h]_k$ in, making the computation much more reasonable than working
with a full basis. We just need to figure out which forms $(f_0\ox\mu)(mz)$
are actually modular for $\Ga_1(N)$. The first step of doing this is to locate
a twist $g$ of $f_0$ which is twist-minimal (i.e. $g$ is not itself a twist
of any lower-level newforms) - this is clearly a finite computation, which we
make some remarks on in Section \ref{sec:ComputingMinimalTwists}. Once we have
$g$ we can determine the level and $p$-th Fourier coefficient of any twist
$g\ox\mu$ of it via a prime-by-prime analysis, either working classically
(as in Section 3 of \cite{AtkinLi1978} and in \cite{Asai1976}) or adelically
(where it's clear what happens if the local component of the representation is
principal series or special, but more complicated if it's supercuspidal; see
the discussion in Sections 2 and 4 of \cite{LoefflerWeinstein2012} and Section
2 of \cite{Humphries2015}). The results of this analysis are summarized in the
following lemma.

\begin{lem} \label{lem:LevelQExpansionTwists}
	Let $g = \s b_n q^n$ be a twist-minimal newform of level $N_g$ and character
	$\ch_g$, and let $N_{g,\ch}$ be the conductor of $\ch_g$. Fix a prime $p$
	and let $p^{r_g}$ be the exact power of $p$ dividing $N_g$, $p^{r_{g,\ch}}$
	the exact power dividing $N_{g,\ch}$, and $\nu$ a Dirichlet character of
	prime-power conductor $p^u$. 
	\begin{itemize}
		\item If we don't have $r_g = r_{g,\ch} > 0$, then $g\ox\nu$ has level
			$\lcm(N_g,p^{2u})$ and equals the naive twist $g_\nu$. 
		\item If $r_g = r_{g,\ch} > 0$ and $u \ne r_{g,\ch}$ then $g\ox\nu$ has
			level $\lcm(N_g,p^{u+r_{g,\ch}},p^{2u})$ and equals $g_\nu$.
		\item If $r_g = r_{g,\ch} > 0$ and $u = r_{g,\ch}$, but the $p$-part of
			the conductor of $\ch_g \nu$ is $p^{r'} > 1$, then $g\ox\nu$ has level
			$\lcm(N_g,p^{u+r'})$ and equals $g_\nu$.
		\item If $r_g = r_{g,\ch} > 0$, $u = r_{g,\ch}$, and $\ch_g \nu$ is
			unramified at $p$, then $g\ox\nu$ has level $N_g$ and does not equal
			the naive twist $g_\nu$; instead it has a coefficient of
			$(\ch_g\nu)(p) \Lb_p$ for $q^p$ and thus can be explicitly written as 
			\[ (g\ox\nu) = \s_{(n,p) = 1} \nu(n) b_n q^n 
					+ \s_{n = p^i n'} (\ch_g\nu)(p)^i \nu(n') \Lb_p^i b_{n'} q^n. \]
	\end{itemize}
\end{lem}
\begin{proof}
	The first case corresponds to the local representation of $g$ at $p$ either
	being unramified, special of level $p$, or supercuspidal. In all three cases
	it's clear that the twisted local representation will result in $g\ox\nu$
	having a trivial $p$-th Fourier coefficient so $g\ox\nu = g_\nu$. In the
	first two cases one can explicitly compute the conductor of the twisted
	local representation to be $p^{2u}$, and for the supercuspidal case we know
	that the conductor will be bounded above by $\max(p^{2u},p^{r_g})$ with
	equality if $2u > r_g$ via Section 3 of \cite{AtkinLi1978}, and equality if 
	$2u \le r_g$ by our assumption of twist-minimality.
	\par
	The remaining type of twist-minimal local representations are principal
	series $\pi(\ch_1,\ch_2)$ where one of the two characters $\ch_i$ is
	ramified; the final three possibilities cover subcases of this situation. In
	any case we know $g\ox\nu$ has local representation
	$\pi(\ch_1\nu_p,\ch_2\nu_p)$ where $\nu_p$ is the local character associated
	to the adelic lift of $\nu$. Here it is clear how to analyze the conductor
	of this principal series representation (since $\ch_1$ is unramified and
	$\ch_1\ch_2$ is the $p$-part of the adelic lift of $\ch_g$, the conductor of
	$\ch_1\nu_p$ is $p^u$ and the conductor of $\ch_2\nu_p$ equals the conductor
	of $\ch_g\nu$). In the case where $\ch_2\nu_p$ is unramified, its value at
	$p$ will give rise to the coefficient of $q^p$ in $g\ox\nu$ which is killed
	off in the naive twist $g_\nu$, and using the relations between the
	characters lets us compute this coefficient to be $(\ch_g\nu)(p) \Lb_p$.
\end{proof}

With this analysis it's easy to come up with a list of twists $g\ox\mu$ of
level at most $Nh$ and moreover find the exact level of each $g\ox\mu$ so
we can determine exactly which oldforms $(g\ox\mu)(mz)$ are of level $Nh$ as
well. This gives us a finite list $g_1,\ld,g_M$ of modular forms of which we
know $f|[\al_h]_k$ is a linear combination of, and we can proceed with a
computation similar to the one of the previous section: we sample at some
collection of more than $M$ points, compute the values of $g_l$ and
$f|[\al_h]_k$ at each point, and use least-squares approximation to find the
best fit for the list of coefficients in the relation $f|[\al_h]_k = \s c_l
g_l$. 
\par
Once again it seems very difficult to establish any sort of rigorous bounds on
the error in this computation, but in practice it works quite well and
heuristically one expects that the error in the computation will be near the
same order of magnitude as where we truncated our sums. More specifically, if w
e normalize all of our values $f|[\al_h]_k(z_j)$ and $g_l(z_j)$ by dividing by
$q_j = \exp(2\pi i z_j)$ and then numerically compute our values
$f|[\al_h]_k(z_j)/q_j$ and $g_l(z_j)/q_j$ to within an error of $10^{-E}$, then
we expect the numerical values of $c_l$ will be such that the product $c_l \dt
(g_l(z_j)/q_j)$ is accurate to about $10^{-E}$ as well. For the $g_l$'s that
are actually newforms, the coefficient of $q$ is $1$ so $g_l(z_j)/q_j \ap 1$,
and thus these $c_l$'s themselves should be accurate to about $10^{-E}$. For
$g_l$'s of the form $(g_0\ox\mu)(mz)$ for $m > 1$, the value of $g_l(z_j)/q_j$
is significantly smaller (approximately $\exp(-2\pi(m-1)\Im(z_j))$) so the
error in $c_l$ might be larger, but we can compensate for this by making our
original computation more accurate (as described in the algorithm below). 
\par
The last thing to decide is what points $z_j$ we want to sample. In this case
we have quite a bit of flexibility, and we are free to pick points $z_j$ to try
to minimize the number of terms needed to be used when computing the values of
our modular forms from the $q$-expansion of $f$ and its twists. Roughly
speaking this amounts to trying to simultaneously minimize both $|\exp(2\pi i
z)|$ and $|\exp(2\pi i (\al_h\dt z))|$, i.e. to simultaneously maximize
$\Im(z)$ and $\Im\pf{ahz+b}{chz+d} = \f{h\Im(z)}{|chz+d|^2}$. Comparing these
we can compute that the best choice for $z$ has $\Im(z) = \rt{h}/\rt{2}c$ and
$\Re(z) = -d/c$; expanding this a bit since we need multiple points we can
calculate that if we choose $z$ in the rectangle 
\[ \Im(z) \in \b[ \f{1}{2c\rt{h}} , \f{1}{c\rt{h}} \e] 
		\qq\qq \Re(z) \in \b[ \f{-d-\rt{h/2}}{ch},\f{-d+\rt{h/2}}{ch} \e] \]
then $|\exp(2\pi i z)|$ and $|\exp(2\pi i \al_h z)|$ are both bounded above by
$\exp(-\pi/c\rt{h})$. 

\begin{algo}[Least-squares for twists of an eigenform]
\label{algo:LeastSquaresTwistsOfEigenform}
	Suppose we have $f = \s a_n q^n \in M_k(N,\ch)$ an eigenform for all
	prime-to-$N$ Hecke operators, and we want to compute its expansion
	$f|[\al_h] = \s b_n q^n$ at a cusp given by a matrix $\al_h$ in our notation
	above. Suppose further that we've fixed constants $E_0$, $K$, and $C$ such
	that for $n\le K$ we would like to compute the coefficient $b_n$ to with
	an error of approximately $10^{-E_0} e^{nC}$. We proceed as follows:
	\begin{itemize}
		\item Determine the newform $f_0$ associated to $f$ and a twist-minimal
			newform $g_0$ that's a twist of $f_0$.
		\item For Dirichlet characters $\mu$ of modulus $N$, determine the level
			of the twist $g_0\ox\mu$; create a list $g_1,\ld,g_L$ of all forms
			$(g_0\ox\mu)(mz)$ that have level $Nh$. 
		\item Pick $M$ random points $z_1,\ld,z_M$ (we use $M = 2L$) with
			$1/2c\rt{h} \le \Im(z) \le 1/c\rt{h}$ and $(-d-\rt{h/2})/ch \le \Re(z)
			\le (-d+\rt{h/2})/ch$. 
		\item Set our truncation point for sums to be when the tail is size
			$10^{-E}$ where $E = E_0 + \f{m_0-1}{\log 10}( 2\pi\f{\rt{h}}{c} - C)$
			(or $E = E_0$, if $2\pi\f{\rt{h}}{c} < C$) where $m_0$ is the largest
			integer $\le K$ such that we have a modular form $(g_0\ox\mu)(m_0 z)$
			on our list.
		\item Numerically compute the values $f|[\al_h](z_j)$ using the
			$q$-expansion for $f$ to accuracy $10^{-E}$, and fill these into a
			vector $b$.
		\item Numerically compute the values $g_l(z_j)$ to an accuracy of
			$10^{-E}$, using the $q$-expansions for the twists as described in
			Lemma \ref{lem:LevelQExpansionTwists}, and fill these into a matrix
			$A$. 
		\item Numerically find the least squares solution to $Ax = b$, which
			approximates the values of $c_1,\ld,c_L$ in our linear combination.
			Use these values plus the $q$-expansions of the $g_l$ to provide a 
			numerical approximation for the $q$-expansion of $f|[\al_h] = \s c_l
			g_l$.
	\end{itemize}
\end{algo}

The change of the truncation point to $10^{-E}$ is to guarantee that we've
computed everything out far enough so that even the coefficient of the (small)
values of $(g_0\ox\mu)(m_0 z)$ can be computed with as much accuracy as we
want. In principle this could go quite far beyond the original accuracy
$10^{-E_0}$ we were interested in, and if this becomes an issue the choice of
points $z_j$ could be adjusted instead. However for most practical purposes the
change is not a serious problem, and the number of terms needed to be computed
in the sums usually stays far below the number needed for the algorithm in the
previous section. 
\par
For an example, we return to the modular form 
\[ f = q - 3q^2 + q^4 - 15q^5 - 25q^7 + 21q^8 + 45q^{10} + \cd 
		\in S_4(\Ga_0(27)) \]
considered in the previous section, and look at the expansion $f|[\al_3]_k$ at
the cusp $1/3$ (with the matrix $\al_1$ considered there). Now we know that
this translate must be a linear combination of twists lying in
$S_k(\Ga_1(81))$. One can check directly that $f$ is twist-minimal and the list
of possible basis elements are 
\[ f(z), f(3z), (f\ox\mu_1)(z), (f\ox\mu_1^2)(z), (f\ox\mu_1^3)(z),
	(f\ox\mu_1^3)(3z), (f\ox\mu_1^4)(z), (f\ox\mu_1^5)(z), \]
where we fix $\mu_1$ to be the Dirichlet character modulo $9$ defined on the
multiplicative generator $2$ of $(\Z/9\Z)^\x$ by $\mu_1(2) = \ze_6$. Then a 
numerical computation finds that $f|[\al_3]_k(z)$ is approximately 
\small
\begin{multline*}
	(.469846310392954 - .171010071662834i)(f\ox\mu_1)(z)
	+ (.469846310392954 + .171010071662834i)(f\ox\mu_1^2)(z) \\
	+ (.469846310392954 - .171010071662834i)(f\ox\mu_1^4)(z)
	+ (-.469846310392954 - .171010071662834i)(f\ox\mu_1^5)(z)
\end{multline*}
\normalsize
(omitting the factors where the numerically-calculated coefficients are very
close to zero). Numerically summing up this linear combination of Fourier
expansions, one gets a numerical $q$-series that (up to our expected error)
agrees with the one computed by our other algorithm in the previous section.
What's more interesting is to try to identify the complex numbers appearing as
coefficients here: they all seem to be approximating $\f12$ times an 18th root
of unity, and suggest that 
\[ f|[\al_3]_k(z) = \f{\ze_{18}^{-1}}{2} (f\ox\mu_1)(z)
		+ \f{\ze_{18}}{2} (f\ox\mu_1^2)(z)
		+ \f{\ze_{18}^{-1}}{2} (f\ox\mu_1^4)(z)
		+ \f{-\ze_{18}}{2} (f\ox\mu_1^5)(z). \]
Combining these $q$-series one can work out explicitly that if $f(z) = \s a_n
q^n$ then $f|[\al_3]_k(z) = -\s \ze_{18}^{8n} a_n q(n)$. So for this example,
the expansion is (up to a scalar) an additive twist of the original
$q$-expansion of $f$. Other modular forms have other behavior; for instance if
we consider the newform and matrix
\[ f = q  - q^2 - 7q^3 - 7q^4 + 7q^6 + 6q^7 + 15q^8 + 22q^9 + \cd 
		\in S_4(\Ga_0(25))  \qq \al_1 = \M{cc}{ 1 & -1 \\ 5 & -4}, \]
the numerically-calculated expansion of $f|[\al_1]_k(z)$ includes nonzero
coefficients for all four twists of $f = \s a_n q^n$ by Dirichlet characters
modulo 5 and can be expressed as $f|[\al_1]_k = \s \xi(n) q^n$ where $\xi$ is
periodic modulo 5 and satisfies 
\small
\begin{align*}
	\xi(1) &\ap -.809016994374947 + 1.11351636441161i 
		\ap \f{\cos(6\pi/5)}{\cos(7\pi/10)} \ze_{20}^7 \\
	\xi(2) &\ap .309016994374947 - .100405707943114i 
		\ap \f{\cos(2\pi/5)}{\cos(\pi/10)} \ze_{20}^{19} \\
	\xi(3) &\ap .309016994374947 + .100405707943114i 
		\ap \f{\cos(2\pi/5)}{\cos(\pi/10)} \ze_{20} \\
	\xi(4) &\ap -.809016994374947 - 1.11351636441161i 
		\ap \f{\cos(6\pi/5)}{\cos(7\pi/10)} \ze_{20}^{13}
\end{align*}
\normalsize
So we have numerically identified the expansion of $f(z)$ at the cusp $1/5$ as
being a ``twist'' of $f$ by a periodic function $\xi$ with coefficients that
are are algebraic numbers in $\Q(\ze_{20})$. We do not pursue a theoretical
understanding of how or why these specific coefficients arise; for our purposes
we just need the numerical values. 


\subsection{When can the eigenspace be narrowed down?} 
\label{sec:NarrowingDownEigenspace}

In this section we refine Lemma \ref{lem:WidthOfCuspForModularForm}, to narrow
down the space in which we can be guaranteed $f|[\al_h]_k$ lives (and
accordingly prune the list of potential twists considered in Algorithm
\ref{algo:LeastSquaresTwistsOfEigenform}). For some cusps the result is close
to optimal already, but for others it fails quite badly - most notably the cusp
$0$ (which always has width $N$), with the matrix 
\[ \al_1 = \M{cc}{ 0 & -1 \\ 1 & 0 }. \]
If $f$ is a newform, Lemma \ref{lem:WidthOfCuspForModularForm} and Theorem
\ref{thm:LinearCombinationOfTwists} can only tell us that $f|[\al_N]_k$ is a
linear combination of twists of $f$ (and their images under degeneracy maps)
which are modular of level $N^2$ for some character. But $|[\al_N]_k$ is just
the Atkin-Lehner operator $W_N$, and the theory of newforms tells us that
$f|[\al_N]_k$ is a scalar multiple of one particular newform $f_\rh$ of level
$N$. So in this case our result is quite far from sharp, and running Algorithm
\ref{algo:LeastSquaresTwistsOfEigenform} naively may take quite some time due
to including a great many unneeded basis elements.
\par
However, the situation is not quite as simple as it seems - while it is true
that for the particular matrix $\al_1$ above that $f|[\al_N]_k$ is always a
scalar multiple of a single newform, this will fail for other choices of
matrices taking $\oo$ to the cusp $0$ (even other ones with $a = 0$ and $c =
1$). The key point is actually that the lower-right entry $d$ is zero; this
makes the behavior of the lower-left and lower-right entries of the conjugate
$\al_h \ga \al_h\1$ sensible and allows us to conclude $f|[\al_h]_k$ transforms
reasonably under $\ga$. So a refinement of Lemma
\ref{lem:WidthOfCuspForModularForm} can only reasonably hold if we are careful
to choose our matrix $\al_1$ carefully. In the case of a general cusp $a/c$,
we'd like the product $cd$ to be as close to divisible by $N$ as possible -
since $c$ and $d$ must be coprime, in particular $d$ should be divisible by the
prime-to-$c$ part of $N$.
\par
So, fix a modular form $f \in M_k(N,\ch)$, and a cusp $a/c$ with associated
width $h$ for $f$ (so $h$ is determined from $N$, $c$, and $\ch$ as in Lemma
\ref{lem:WidthOfCuspForModularForm}). Factor $N$ as $c_0\dt d_0$, where $c_0$
and $c$ have the same prime divisors and $d_0$ is coprime to $c_0$. Then choose
a matrix 
\[ \al_1 = \M{cc}{ a & b \\ c & d } \in \SL_2(\Z) \]
with $a$ and $c$ as in our cusp ($\al_1$ takes $\oo$ to $a/c$) and with $d$
divisible by $d_0$ (which we can do because $d_0$ is coprime to $c$). Factor
the width $h$ as $h_c\dt h_d$ where $h_c | c_0$ and $h_d | d_0$, and also let
$\ch_c$ and $\ch_d$ denote the restrictions of $\ch$ to $\Z/c_0\Z$ and
$\Z/d_0\Z$, respectively. 

\begin{prop}
	In the above setup, for any matrix 
	\[ \ga = \M{cc}{ A & B \\ C & D } \in \Ga_0(Nh_c) \]
	satisfying $(A-D)c \ee 0 \md{c_0}$ we have $(f|[\al_h]_k)|[\ga]_k =
	(\ch_c\ch_d\1)(D) f|[\al_h]_k$. 
\end{prop}
\begin{proof}
	We have $f|[\al_h]_k|[\ga]_k = f|[\al_h\ga\al_h\1]_k|[\al_h]_k$, so we 
	want to show that $f$ transforms under $\al_h\ga\al_h\1$ by the scalar
	$(\ch_c\ch_d\1)(D)$. An explicit computation gives  
	\[ \al_h\ga\al_h\1 
			= \M{cc}{ * & * \\ 
						-Bc^2 h + (A-D)cd + Cd^2/h & Bach + Dad - Abc - Cbd/h }. \]
	By construction of $d$ and assumption that $C \ee 0\md{Nh_c}$ and $(A-D)c
	\ee 0 \md{c_0}$ we can conclude that each of the three terms in the
	lower-left entry are divisible by $N$, and thus $\al\ga\al_h\1 \in
	\Ga_0(N)$. Thus $f|[\al_h\ga\al_h\1]$ equals $\ch(Bach + Dad - Abc - Cbd/h)
	\dt f$. 
	\par
	So we just need to simplify 
	\[ \ch(Bach + Dad - Abc - Cbd/h) = \ch(D + (D-A)bc + Bach), \]
	where we can remove $Cbd/h \ee 0\md N$ immediately. To further work with 
	this we split $\ch$ as our product $\ch_c\ch_d$. Since $\ch_c$ is defined
	modulo $c_0$ and $(D-A)bc \ee 0\md{c_0}$ we have 
	\[ \ch_c(D + (D-A)bc + Bach) = \ch_c(D + Bach) = \ch_c(D) \ch_c(1 + D'Bach)
		= \ch_c(D) \]
	with the last equality because $h$ is defined so that $\ch$ (and thus
	$\ch_c)$ is trivial on $1 + ch\Z$. Similarly since $\ch_d$ is defined modulo
	$d_0$ and we have  
	\[ \ch_d(D + (D-A)bc + Bach) = \ch_d(D + (D-A)(-1)) 
			= \ch_d(A) = \ch_d\1(D), \]
	using that working modulo $d_0$ we have $ch \ee 0$, $bc \ee -(ad-bc) =
	-1$, and $AD \ee AD-BC = 1$. 
\end{proof}

So by properly choosing the lower-right entry $d$ in our matrix $\al_1$,
we can guarantee that $f|[\al_h]_k$ is actually modular of level $Nh_c$ rather
than just $Nh$, and moreover get at least some control of the character. In the
case that $c$ and $N/c$ are coprime (i.e. $c = c_0$ in our notation above),
this proposition fully determines the character and states that $f|[\al_h]_k$
lies in $M_k(Nh_c,\ch_c\ch_d\1)$, but in the general case where there are
primes dividing both $c$ and $N/c$ we can only give a transformation rule for
matrices $\ga$ such that $A \ee D \md{c_0/c}$, i.e. for some intermediate
congruence group $\Ga_H(Nh_c)$. This allows $f|[\al_h]_k$ to be a linear
combination of forms with characters that agree with $\ch_c\ch_d\1$ on $H$ -
and based on numerical examples in such cases this seems to be the best one
could hope for. 
\par
We can restate the result of our computation as follows, which we can view as a
strengthened version of Lemma \ref{lem:WidthOfCuspForModularForm} but which
only applies if the matrix $\al_1$ has a ``correctly-chosen'' lower-right
entry.

\begin{prop} \label{prop:RefinementOfWhereModularFormLands}
	Fix a modular form $f\in M_k(N,\ch)$ and a cusp $a/c$ with width $h$ (for
	$f$). Choose a matrix $\al_1 \in \SL_2(\Z)$ taking $\oo$ to $a/c$ with
	bottom-right entry $d$ divisible by the prime-to-$c$ part of $N$. Then we 
	have 
	\[ f|[\al_h]_k \in \Op_{\ch'} M_k(Nh_c,\ch') \]
	where $\ch'$ runs over all characters of $(\Z/N\Z)^\x$ which agree with
	$\ch_c \ch_d\1$ on the subgroup $H \le (\Z/N\Z)^\x$ which is the kernel of 
	the map $(\Z/N\Z)^\x \to (\Z/\f{c_0}{c}\Z)^\x$ given by $a+N\Z \mt a^2 +
	\f{c_0}{c}\Z$. 
\end{prop}

Suppose $p$ is a prime with exact power $p^m$ dividing $N$, and let $\ch_p$
denote the $p$-component of $\ch$ (a character of $(\Z/p^m\Z)^\x$). If $p\nd c$
(i.e. $p|d$) the restriction on $\ch'$ above requires that $\ch'_p = \ch_p\1$.
For $p|c$ we need to consider the exact power $p^{m'}$ dividing $c_0/c$, and
the restriction is that $\ch'_p$ agrees with $\ch_p$ on the multiplicative
subgroup $1+p^{m'}\Z$ (and that $\ch'_p$ and $\ch_p$ have the same sign, but
this is determined by the parity of $k$ anyway). 
\par
To apply this in Algorithm \ref{algo:LeastSquaresTwistsOfEigenform}, we need to
consider which twists $(f\ox\mu)$ will lie in the space considered above. Since
the character of the twist is $\ch\mu^2$ we see that for $p|d$ we need
$\ch_p\mu_p^2 = \ch_p\1$, i.e. that $\mu_p = \ch_p\1$ up to a quadratic
character. For $p|c$ we need that $\ch_p\mu_p^2 = \ch_p$ on $1+p^{m'}\Z$, i.e.
that $\mu_p^2$ is trivial modulo $p^{m'}$. In the case when $p$ is odd, if $m'
= 0$ this requires $\mu_p$ to be either trivial or the unique quadratic
character, while if $m' \ge 1$ then $\mu_p$ must have conductor at most
$p^{m'}$. For $p = 2$, if $m' = 0,1$ then $\mu_p$ may be trivial or any of the
four quadratic characters, while if $m' \ge 2$ then $\mu_p$ may be any
character of conductor at most $2^{m'+1}$.
\par
So Proposition \ref{prop:RefinementOfWhereModularFormLands} represents a
significant restriction of potential twists appearing in Algorithm
\ref{algo:LeastSquaresTwistsOfEigenform} compared to our original result from
Lemma \ref{lem:WidthOfCuspForModularForm}. However, it still does not recover
the full strength of what newform theory tells us for the cusp $0$ (where we're
applying the Atkin-Lehner involution $W_N$) or the full strength of Asai's
result \cite{Asai1976} covering the case when $N$ is squarefree, because our
result cannot see distinguish differences by quadratic characters. However, we
remark that combining Proposition \ref{prop:RefinementOfWhereModularFormLands}
with the restriction to level $Nh_c$ often rules out incorrect twists; for
instance in the squarefree case we know $f|[\al_h]_k$ will still have $N$ and
this rules out most incorrect twists because they would have a higher level.
\par
We suspect that one could analyze how Hecke operators interact with
our slash operators $|[\al_h]_k$ (similarly to what is done in Chapter 4.6 of
\cite{MiyakeModForms}, or in \cite{Asai1976}) and further narrow down the list
of twists needed to be considered in Algorithm
\ref{algo:LeastSquaresTwistsOfEigenform}; we do not attempt to carry this out
here.
 
\section{Theoretical results on transferring modular forms to other cusps}
\subsection{Transformations of eigenforms to other cusps}
\label{sec:TransformEigenformToOtherCusp}

In this section we prove Theorem \ref{thm:LinearCombinationOfTwists}, that if
$f \in S_k(N,\ch)$ is an eigenform of all Hecke operators $p\nd N$, then 
the translate to another cusp $f|[\al_h]_k \in S_k(\Ga_1(Nh))$ arises as a
linear combination of twists of $f$ (and their images under degeneracy maps).
To begin our analysis we split $[\al_h]_k$ into its two parts
\[ \begin{tikzcd}
	S_k(N,\ch) \a{rr}{[\al_1]_k} 
		&& S_k(\Ga_1(N,h)) \a{rr}{[\ta_h]_k} 
		&& S_k(\Ga_1(Nh))
\end{tikzcd}. \]
To study this we recall the general definition of Hecke operators on these
spaces. We can consider congruence subgroups of the form
\[ \Ga_H(N,n) = \b\{ \M{cc}{ a & b \\ c & d } 
		: c \ee 0\md N,\ b\ee 0\md n,\ a+n\Z,d+n\Z \in H \e\} \se \SL_2(\Z) \]
for $n|N$ and $H$ a subgroup of $(\Z/n\Z)^\x$. For such a subgroup $\Ga =
\Ga_H(N,n)$ we set 
\[ \De = \De_H(N,n) = \b\{ \M{cc}{ a & b \\ c & d } 
	: c \ee 0\md N,\ b\ee 0\md n,\ a+n\Z \in H,\ ad-bc > 0 \e\} \se M_2(\Z). \]
For $m > 0$ we take the subset $\De_m = \{ \be\in\De : \det(\be) = m \}$. Then,
for $\ch$ a character of $H \le (\Z/N\Z)^\x$ that we view as a character of
$\De$ by acting on the upper-left entry $a$, we define the Hecke operator $T_m$
on $M_k(\Ga,\ch)$ by taking a decomposition of $\De_m$ in terms of left cosets
of $\Ga$: 
\[ \De_m = \dU_i \Ga \be_i, \qq\qq 
	T_m f = m^{k/2-1} \s_i \ch(\be_i) f|[\be_i]_k. \]
The theory of Hecke operators is worked out in this generality in Chapter 3 of
\cite{ShimuraAutForms}. In particular, Proposition 3.36 gives an explicit
formula for $T_m$ that lets us conclude that passing to a larger congruence
subgroup preserves Hecke operators prime to the level. 

\begin{prop}
	Suppose we have two subgroups of the above form satisfying $\Ga_{H'}(N',n')
	\le \Ga_H(N,n)$ (which implies $N | N'$, $n | n'$, and the pullback of $H$
	to $(\Z/N'\Z)^\x$ contains $H'$); for any character $\ch$ of $H$ (and its
	corresponding restriction $\ch'$ to $H'$) we have an inclusion
	\[ M_k(\Ga_H(N,n),\ch) \se M_k(\Ga_{H'}(N',n'),\ch'). \]
	If $m$ is an integer prime to $N'$, then the Hecke operators $T_m$ on these
	two spaces are compatible with the inclusion map. 
\end{prop}

With this setup it's easy to check that our map
$[\ta_h]_k$ is compatible with Hecke operators $T_m$ for $(m,N) = 1$. 

\begin{lem}
	Fix an integer $N$ and a divisor $h$ of it, and consider the map 
	\[ [\ta_h]_k : S_k(\Ga_1(N,h)) \to S_k(\Ga_1(Nh)). \]
	Then if $(m,N) = 1$ the Hecke operators $T_m$ on each space are compatible
	with $[\ta_h]_k$: we have $T_m(f|[\ta_h]_k) = (T_m f)|[\ta_h]_k$ for all
	$f$.
\end{lem}
\begin{proof}
	One can check that conjugation by $\ta_h$ takes $\Ga_1(N,h)$ to $\Ga_H(Nh)$
	for $H\le (\Z/hN\Z)^\x$ the kernel of the projection to $(\Z/N\Z)^\x$, and
	also $\ta_h\1 \De_1(N,h)_m \ta_h = \De_H(Nh)_m$. From this we can see that
	$[\ta_h]_k$ maps from $S_k(\Ga_1(N,h))$ to $S_k(\Ga_H(Nh))$ and preserves
	$T_m$, and we can include into $S_k(\Ga_1(Nh))$.
\end{proof}

On the other hand, the interaction of $[\al_1]_k$ with Hecke operators seems
less well-known. In trying to analyze this we run into the problem that
$f|[\al_1]_k$ is invariant under the subgroup $\al_1\1 \Ga_0(N) \al_1$ which is
hard to identify and may not be one of the types of subgroups we've already
studied. We can always find a congruence subgroup inside of it that is (what
we've proven is that $\Ga_1(N,h)$ is contained in $\al_1\1 \Ga_0(N) \al_1$),
but there isn't a direct link between the Hecke operators involved. However, we
can see that $[\al_1]_k$ is compatible with some of the Hecke operators as
follows. 

\begin{prop}
	For a matrix $\al_1 \in \SL_2(\Z)$ and the associated integer $h$ as above,
	the operator $[\al_1]_k : S_k(N,\ch) \to S_k(\Ga_1(N,h))$ is compatible with
	Hecke operators $T_m$ for $m \ee 1\md N$.
\end{prop}
\begin{proof}
	Consider the following diagram of spaces of modular forms: 
	\[ \begin{tikzcd}
		S_k(N,\ch) \a{rr}{[\al_1]_k} \a[in]{d}{} 
			&& S_k(\Ga_1(N,h)) \a[in]{d}{} \\
		S_k(\Ga(N)) \a[s]{rr}{[\al_1]_k} && S_k(\Ga(N)) 
	\end{tikzcd}; \]
	we've already established that $[\al_1]_k$ defines a map between the top two
	spaces, and it clearly also defines one between the bottom two spaces
	because $\Ga(N)$ is normal in $\SL_2(\Z)$. The diagram evidently commutes
	because the operator $[\al_1]_k$ is defined independently of the ambient
	space it's used on. To prove that $T_m$ is compatible with the top map
	$[\al_1]_k$, it's sufficient to prove it's compatible with the bottom one
	and use compatibility of the vertical inclusions. 
	\par
	So we want to prove that for $m\ee 1\md N$, the endomorphisms $T_m$ and
	$[\al_1]_k$ on $S_k(\Ga(N))$ commute. For this, note that by definition
	$\De(N)_m$ is all matrices 
	\[ \de = \M{cc}{ A & B \\ C & D } \ee \M{cc}{ 1 & 0 \\ 0 & * } \md N 
			\qq\qq AD - BC = m; \]
	if $m \ee 1\md N$ then this forces $D \ee 1 \md N$ and thus $\de \ee I \md
	N$. Then conjugating such a $\de$ by $\al_1$ gives another matrix congruent
	to $I$ modulo $N$, and we conclude that conjugation by $\al_1$ is an
	automorphism of $\De(N)_m$. Thus if $\De(N)_m = \dU \Ga(N)\be_i$ is a coset
	decomposition, conjugating gives that $\De(N)_m = \dU
	\Ga(N)\al_1\1\be_i\al_1$ is also one, and $T_m$ can be written in terms of
	either, and thus 
	\[ T_m(f|[\al_1]_k) = m^{k/2-1} \s_i f|[\al_1]|_k[\al_1\1\be_i\al_1]_k
			= m^{k/2-1} \s_i f|[\be_i]_k[\al_1]_k = (T_m f)|[\al_1]_k. \QED \]
\end{proof}

Putting things together we have: 

\begin{thm} \label{thm:TranslatingPreservesSomeHeckeOperators}
	The operator $[\al_h]_k : S_k(N,\ch) \to S_k(\Ga_1(Nh))$ is compatible with
	the Hecke operators $T_m$ defined on both spaces for $m\ee 1\md N$. Thus, if
	$f_0 = \s a_n q^n \in S_k(N_0,\ch)$ is a newform of level $N_0|N$ and $f
	\in S_k(N,\ch)$ is anything lying in the corresponding prime-to-$N$
	eigenspace, then $f|[\al_h]_k$ satisfies $T_m(f|[\al_h]_k) = a_m
	f|[\al_h]_k$ for $m \ee 1\md N$. 
\end{thm}

So if we start with $f$ an eigenform of the prime-to-$N$ Hecke algebra on
$S_k(N,\ch)$ (associated to a newform $f_0 = \s a_n q^n$ but perhaps itself
an oldform), then $f|[\al_h]_k$ is a ``partial'' eigenform lying in the
subspace \[ \{ g\in S_k(\Ga_1(Nh)) : T_m(g) = \la_m g, m\ee 1\md N \}. \] This
subspace breaks up as a direct sum of prime-to-$N$ eigenspaces, each of which
is associated to some newform $g_{0,i}$. The next theorem lets us pin down
these $g_{0,i}$'s as being twists of $f$. 

\begin{thm} \label{thm:PartiallyAgreeingNewformsAreTwists}
	Suppose $f_0 = \s a_n q^n$ is a newform, and $g_0 = \s b_n q^n$ is another
	newform such that $a_m = b_m$ for $m\ee 1\md N$. Then $g_0$ is a twist
	$f_0\ox\mu$ for some Dirichlet character $\mu$ modulo $N$.
\end{thm}
The idea is essentially to define $\mu(m+N\Z) = b_m/a_m$ and check that this is
independent of the representative of $m$ and defines a Dirichlet character. If
we have plenty of coefficients where $a_m \ne 0$ then this makes sense and the
argument goes through easily (claims 2 and 3 below are the main idea); making
the argument go through for forms where we may have many $a_m$'s equal to zero
just requires a little more care.
\begin{proof}
	Define a subset $H \se (\Z/N\Z)^\x$ consisting of all residue classes
	$c+N\Z$ such that there exists an infinite set $\{ l_i \}$ of
	representatives of $c+N\Z$ with the $l_i$ pairwise coprime and
	satisfying $a_{l_i} \ne 0$. Note that for any given integer $L$, all but
	finitely many elements of $\{ l_i \}$ will be coprime to $L$. (In most cases
	we'd expect to be able to take infinitely many primes $p\ee c\md N$ with
	$a_p \ne 0$ as such a set). 
	\par
	\E{Claim 1: $H$ is a subgroup.} Suppose we have two residue classes $c+N\Z$
	and $c'+N\Z$ satisfying our condition, with infinite sets $\{ l_i \}$ and
	$\{ l'_j \}$. Then $\{ l_i l'_j : (l_i,l'_j) = 1 \}$ is a set of
	representatives of $cc' + N\Z$ with $a_{l_il'_j} = a_{l_i} a_{l'_j} \ne 0$
	for all of its elements, and there exists an infinite subset of it that's
	pairwise coprime (for any finite subset that's pairwise coprime, we only
	have finitely many $l_i$'s and $l'_j$'s involved, and this only throws away
	a finite list of possible things to add, so a maximal such subset must be
	infinite). 
	\par
	\E{Claim 2: We have a well-defined function $\mu : H \to \C$ given by setting
	$\mu(c+N\Z) = b_l/a_l$ for any $l \in c+NZ$ with $a_l \ne 0$}. By
	assumption there exist plenty of such $l$'s, so we need to check that if
	$l,l' \in c + N\Z$ satisfy $a_l, a_{l'} \ne 0$ then $b_l/a_l =
	b_{l'}/a_{l'}$. By claim (1) we know $H$ is a subgroup so $(c+N\Z)\1$
	satisfies our assumption, and thus we can pick $l'' \in (c+N\Z)\1$ which
	is coprime to both $l$ and $l'$. Then $ll'' \ee 1\md N$ so we have 
	\[ b_l b_{l''} = b_{ll''} = a_{ll''} = a_l a_{l''} \ne 0 \]
	giving $b_l/a_l = a_{l''}/b_{l''}$. An identical computation says
	$b_{l'}/a_{l'} = a_{l''}/b_{l''}$ too. 
	\par
	\E{Claim 3: $\mu$ is a multiplicative character on H.} For two cosets
	$c+N\Z$ and $c'+N\Z$ in $H$, by assumption we can pick representatives
	$l \in c+N\Z$ and $l' \in c'+N\Z$ with $l,l'$ coprime and $a_l,a_{l'}
	\ne 0$. Then we have 
	\[ \mu(cc') = \f{a_{ll'}}{b_{ll'}} = \f{a_l a_{l'}}{b_l b_{l'}} 
			= \mu(c) \mu(c'). \]
	\par
	\E{Claim 4: $\mu$ extends to a Dirichlet character on $(\Z/N\Z)^\x$.} Since
	we have a character $H \to \C^\x$ on a subgroup $H$ of an abelian group
	$(\Z/N\Z)^\x$, it's a general fact that we can extend it to a character of
	the full group. 
	\par
	\E{Claim 5: For every prime $p$ lying in a residue class $c + N\Z$ in $H$,
	we have $b_p = \mu(p)a_p$}. If $a_p \ne 0$ then this is immediate from
	definition of $\mu(p)$. If $a_p = 0$ then picking some $l \in (c + N\Z)\1$
	with $p\nd l$ and $a_l \ne 0$ gives $b_p b_l = b_{pl} = a_{pl} = a_p a_l$
	which forces $b_p = 0 = \mu(p) a_p$.
	\par
	\E{Claim 6: For all but finitely many primes $p$ lying in a residue class
	$c + N\Z$ not in $H$, we have $b_p = a_p = 0$.} The set of $p \in c +
	N\Z$ with $a_p \ne 0$ is certainly finite, since otherwise $c + N\Z$ would
	be in our set $H$ by definition. The set of $p$ with $b_p \ne 0$ but $a_p =
	0$ can't be as large as the order of $c + N\Z$ in $(\Z/N\Z)^\x$, since if
	we had distinct primes $p_1,\cd,p_f$ with $b_{p_i} \ne 0$ and $p_1\cd p_f
	\ee 1\md N$ then we'd have 
	\[ 0 = a_{p_1}\cd a_{p_f} = a_{p_1\cd p_f} 
			= b_{p_1\cd p_f} = b_{p_1} \cd b_{p_f} \ne 0. \]
	\par
	\E{Claim 7: The newform $f_0\ox\mu$ equals $g_0$.} By strong multiplicity
	one, it's sufficient to check that these newforms have the same coefficients
	for all but finitely many primes $p$. For primes $p\nd N$, $f_0\ox\mu$ and
	$g_0$ have $p$-th coefficients $\mu(p)a_p$ and $b_p$, respectively, and
	combining claims 5 and 6 we've verified that all but finitely many of these
	are equal. 
\end{proof}

Combining Theorem \ref{thm:TranslatingPreservesSomeHeckeOperators} and Theorem 
\ref{thm:PartiallyAgreeingNewformsAreTwists} establishes Theorem
\ref{thm:LinearCombinationOfTwists}. 
 
\section{Computing the self-Petersson inner product and comparing to the
adjoint $L$-function}
\label{sec:ComputingPeterssonInnerProducts}
\subsection{Computing the Petersson inner product numerically} 

In this section we describe how to numerically compute the Petersson inner
product of two modular forms $f,g \in S_k(N,\ch)$, given $q$-expansions of both
at $\oo$. This is done by applying a formula of Nelson \cite{Nelson2015} that
expresses the Petersson inner product in terms of the Fourier expansions of $f$
and $g$ at all cusps, combined with our methods for computing these Fourier
expansions. To start, we state the definition of the Petersson inner product
we'll be working with: 

\begin{defn}
	Let $f,g$ be two cusp forms of level $k$ (or even one cusp form and one
	modular form) for congruence subgroups of $\SL_2(\Z)$. If $\Ga$ is any
	congruence subgroup for which both are modular, we define their
	\E{normalized Petersson inner product} as 
	\[ \g{f,g} = \f{1}{\vol(\DH\q\Ga)} 
			\S_{\DH\q\Ga} f(x+iy) \L{g(x+iy)} y^k \f{dx\ dy}{y^2}. \]
\end{defn}

Here $y^{-2}dx\ dy$ is the standard volume measure on the upper half-plane. Our
normalization is by $\vol(\DH\q\Ga) = \f{\pi}{3} [\PSL_2(\Z):\LGa]$, and allows
the definition to be independent of the choice of congruence subgroup $\Ga$
that we view both forms as modular with respect to. Notation in the literature
varies, with some places defining $\g{f,g}$ without this normalizing factor,
and others simply using the index $[\PSL_2(\Z):\LGa]$ rather than the volume of
$\DH\q\Ga$. 
\par
Nelson's formula (Theorem 5.6 of \cite{Nelson2015}) applies to quite general
integrals on modular curves, and our methods for computing Fourier coefficients
at all cusps could be applied to many situations. For the purposes of this
paper we are interested in Petersson inner products, so we specialize the
formula to that case (see Example 5.7 of Nelson's paper): 

\begin{thm}[Nelson] \label{thm:PeterssonFormula}
	Suppose $f = \s_n a_n q^n$ and $g = \s b_n q^n$ are two cusp forms in
	$S_k(N,\ch)$. Then we have 
	\[ \g{f,g} = \f{4}{\vol(\DH\q\Ga)} 
			\s_s \f{h_{s,0}}{h_s} \s_{n=1}^\oo \f{a_{n,s} \Lb_{n,s}}{n^{k-1}} 
				\s_{m=1}^\oo \pf{x}{8\pi}^{k-1} \7( x K_{k-1}(x) - K_{k-2}(x) \7) 
				\qq\qq x = 4\pi m\rt{\f{n}{h_s}}, \]
	where $K_v$ is a $K$-Bessel function, $s$ runs over all cusps of $\Ga_0(N)$,
	$h_{s,0}$ is the width of that cusp for $\Ga_0(N)$, $h_s$ is the width for
	that cusp for $f$ as described in Lemma \ref{lem:WidthOfCuspForModularForm},
	and we choose a single matrix $\al_1$ taking $\oo$ to $s$ and write
	$f|[\al_{h_s}]_k = \s a_{n,s} q^n$ and $g|[\al_{h_s}]_k = \s b_{n,s} q^n$.
\end{thm}

So to apply this formula we just need to compute the Fourier expansions of $f$
and $g$ at each cusp, via our methods from Section
\ref{sec:NumericalExpansionAtCusps}. Since Bessel functions decay exponentially
in their arguments, this matches up well with our algorithms returning Fourier
coefficients with accuracy up to an exponentially decaying factor, and thus
makes it so that each term of the sum over $n$ has an absolute error on the
order of whatever magnitude we want to specify. We can implement this as
follows:

\begin{algo}[Petersson inner product of two modular forms of level $N$]
\label{algo:PeterssonForCuspFormsSameLevel}
	Let $f,g \in S_k(N,\ch)$ be two cusp forms of the same level and character.
	Then we can compute their Petersson inner product to an approximate accuracy
	of $10^{-E}$ as follows: 
	\begin{itemize}
		\item List all of the cusps $s$ of $\Ga_0(N)$, the widths $h_{s,0}$ for
			$\Ga_0(N)$, and their widths $h_s$ of Lemma
			\ref{lem:WidthOfCuspForModularForm}. 
		\item Iterate over cusps $s$, and for each do the following:
		\begin{itemize}
			\item Iterate over $n$ and compute the inner sum over $m$ that
				involves Bessel functions (we'll denote this sum $S_{s,n}$); each
				sum over $m$ can be truncated when the terms get some safe factor
				smaller than $10^{-E}$. Record these sums $S_{s,n}$ for each $n$,
				until we reach some $n_s$ where $S_{s,n}$ is a safe factor smaller
				than $10^{-E}$.
			\item Use one of our previous algorithms to compute the Fourier
				expansions of $f$ and $g$ at the cusp $s$, with absolute accuracy
				$10^{-E}$, relative decay $C$ chosen so that $e^{-Cn} \ge S_n$
				for all $n$, and number of terms desired equal to the number of
				terms $n_s$ we found in the previous step. (Of course for the cusp
				$\oo$ we can skip this and use the Fourier expansion directly).
			\item Compute the products $a_{n,s} \Lb_{n,s}/n^{k-1} \dt S_{s,n}$ and
				sum them up from $n = 1$ to $n_s$. This is the contribution of the
				cusp $s$ to our formula for the Petersson inner product.
		\end{itemize}
		\item Add up the contributions for all cusps $s$, and normalize by the
			constant at the front of the formula.
	\end{itemize}
\end{algo}

For the case we're ultimately interested in, we'll work with three modular
forms natively of different levels; there we want to compute Fourier expansions
for each at their native level to avoid any redundant computation. This is
discussed in Section \ref{sec:EigenformsDifferentLevels}. For modular forms
natively of the \E{same} level, the main case of interest is when $f = g$ are
the same newform; in this case the self-Petersson inner product $\g{f,f}$ is
related to an adjoint $L$-value. 
\par
In fact the standard way to compute $\g{f,f}$ is by way of computing this
$L$-value instead, and we cannot claim our algorithm will be a better way.
Instead, we can use the relation of $\g{f,f}$ with the special value $L(\ad
f,1)$ to provide some numerical verification that Algorithm
\ref{algo:PeterssonForCuspFormsSameLevel}, serving as an introduction to the
sort of comparisons we'll be making in Section \ref{sec:ExplicitIchino}. 


\subsection{Comparing with adjoint $L$-values}
\label{sec:ComparisonWithAdjointLValues}

It is well-known that if $f$ is a newform, its self-Petersson product $\g{f,f}$
is related to a value of the adjoint $L$-function associated to $f$ (or of its
shift, the symmetric square $L$-function for $f$). This is found in papers of
Shimura and Hida (see \cite{Shimura1976}, Section 5 of \cite{Hida1981}, and
Section 10 of \cite{Hida1986}), following ideas going back to Petersson; if one
considers the automorphic adjoint $L$-function $L(\ad f,s)$ defined correctly
at all factors and uses the normalization of the Petersson inner product we do,
the identity can be written as 
\[ L(\ad f,1) = \f{\pi^2}{6} \f{(4\pi)^k}{(k-1)!} \g{f,f} \p_p (*)_p \]
where $(*)_p$ is an explicit factor for primes $p$ dividing the level of $N$
which we will describe soon. In this section we'll recall how to numerically
compute the adjoint $L$-value, and then show several examples where we
numerically compare both sides of this formula and see that our method for
computing $\g{f,f}$ returns the correct results. 
\par
An efficient algorithm for computing values of $L$-functions has been given by
Dokchitser \cite{Dokchitser2004}, which is implemented in SageMath
\cite{sagemath}. This algorithm relies on the functional equation for the
$L$-function in question, and thus requires knowledge of various parameters for
the functional equation in addition to the coefficients (or equivalently, the
Euler factors) of the $L$-function itself. In the case of $L(\ad f,s)$ some of
the parameters are easy: the weight is $1$ (i.e. the functional equation
relates $s$ and $1-s$), the gamma factor is
$\Ga\pf{s+1}{2}\Ga\pf{s+k-1}{2}\Ga\pf{s+k}{2}$, and the sign $\ep$ is always
$+1$. The analytic conductor $N_{\ad}$ of the functional equation, however, is
more subtle to determine. 
\par
The determination of the analytic conductor $N_{\ad} = \p_p N_{\ad,p}$ is a
local problem that needs to be solved at each bad prime $p$, as is the
determination of the correct Euler factor $L_p(\ad f,s)$ and the correction
factor $(*)_p$ in our formula above. This breaks into a case-by-case analysis
based on the local representation $\pi_{f,p}$ of the automorphic representation
associated with $f$. Actually, since $L(\ad f,s)$ is invariant under replacing
$f$ by a twist, the first step is to replace $f$ by a twist $g$ which is
twist-minimal and proceed with analyzing the newform $g$ which has level $N_g$
and character $\ch_g$ of conductor $N_{\ch,g}$. Let $p^r$ be the exact power of
$p$ dividing $N_g$, and $p^{r_\ch}$ the exact power of $p$ dividing
$N_{g,\ch}$. Then we have:
\begin{itemize}
	\item If $r = 0$ (i.e. $p\nd N_g$, even if we have $p|N$) our $L$-function
		is unramified: $N_{\ad,p} = 1$ and the ``good'' Euler factor is $L_p(\ad
		f,s) = L_p(\ad g,s) = (1 - \f{\al_p}{\be_p} p^{-s})\1 (1 -
		\f{\be_p}{\al_p} p^{-s})\1 (1 - p^{-s})\1$ where $\al_p,\be_p$ arise from
		$(X^2 - a_p X + \ch(p) p^{k-1}) = (X - \al_p)(X - \be_p)$. 
	\begin{itemize}
		\item If $p\nd N$ (i.e. $f$ is twist-minimal at $p$) then $p$ is a good
			prime so $(*)_p$ doesn't need to be defined, but if $p|N$ ($f$ is not
			twist-minimal at $p$) we have $(*)_p = (1+1/p) L_p(\ad f,1)$. 
	\end{itemize}
	\item If $r = 1$ but $r_\ch = 0$, the local representation of $g$ at $p$ is
		an unramified special representation and we have $N_{\ad,p} = p^2$, 
		$L_p(\ad f,s) = (1 - \f{1}{p} p^{-s})\1$. 
	\begin{itemize}
		\item If $p\| N$ ($f$ is twist-minimal at $p$) then $(*)_p =
			(1+\f{1}{p})$, while if $p^2 | N$ ($f$ is not twist-minimal at $p$)
			then $(*)_p = (1+\f{1}{p})(1-\f{1}{p^2})\1$.
	\end{itemize}
	\item If $r = r_\ch \ge 1$, the local representation of $g$ at $p$ is a
		half-ramified principal series and we have $N_{\ad,p} = p^{2r}$ and
		$L_p(\ad f,s) = (1 - p^{-s})\1$. 
	\begin{itemize}
		\item If $p^r \| N$ ($f$ is twist-minimal at $p$) then $(*)_p =
			(1+\f{1}{p})$, while if $p^{r+1} | N$ ($f$ is not twist-minimal at
			$p$) then $(*)_p = (1+\f{1}{p})(1-\f{1}{p})\1$.
	\end{itemize}
	\item If $r \ge 2$ and $r > r_\ch$, $\pi_{g,p}$ is supercuspidal. At this
		point it becomes harder to give a clean description of all of our
		quantities, but $N_{\ad,p} = p^e$ for some $e \le 2r$, and the
		$L$-function splits up into two cases: 
		\begin{itemize}
			\item If $\pi_{g,p} \iso \et\ox\pi_{g,p}$ for $\et$ the unramified
				quadratic character of $\Q_p^\x$, then $L(\ad f,s) = (1 +
				p^{-s})\1$ and $(*)_p = 1$.
			\item If $\pi_{g,p} \niso \et\ox\pi_{g,p}$, then $L(\ad f,s) = 1$ and
			$(*)_p = (1+1/p)$.
		\end{itemize}
\end{itemize}

In the supercuspidal case we have not described how to full determine
$N_{\ad,p}$ nor how to determine if $\pi_{g,p} \iso \et\ox\pi_{g,p}$, though in
principle this can be done by the algorithm of Loeffler-Weinstein
\cite{LoefflerWeinstein2012} which explicitly determines $\pi_{g,p}$. In the
case of central trivial character, Nelson-Pitale-Saha
\cite{NelsonPitaleSaha2014} give a finer characterization of the conductor in
Proposition 2.5. In any case we remark that Dokchitser's algorithm gives a way
to numerically check the functional equation for any guesses of $N_{\ad,p}$ and
$L(\ad f,s)$, so one can always recover the correct values that way.
\par
We give some examples of the resulting computations and comparisons. For $f_1 =
\De$, the $\De$-function of weight 12 and level 1 (which has no bad places), we
compute
\[ \f{L(\ad f_1,1)}{\g{f_1,f_1}} 
	\ap \f{0.6317929457\ld}{9.8869793538\ld\dt 10^{-7}} \ap 639015.136088\ld
	\ap \f{\pi^2}{6} \f{(4\pi)^{12}}{11!}. \]
For $f_2 = q - 6q^2 + 9q^3 + 4q^4 + 6q^5 + \cd$ the unique newform of weight 6
and level 3 with trivial character, the local representation at $3$ is special
and we get 
\[ \f{L(\ad f_2,1)}{\g{f_2,f_2}} 
	\ap \f{0.9879391307\ld}{0.00001372666446\ld} \ap 71972.2648922\ld
	\ap \f{\pi^2}{6} \f{(4\pi)^6}{5!} \b( 1 + \f{1}{3} \e) . \]
On the other hand, the twist $f'_2 = q + 6q^2 + 4q^4 - 6q^5 + \cd$ of weight 6
and level 9 has the same $L$-value but the Petersson inner product differs
\[ \f{L(\ad f'_2,1)}{\g{f'_2,f'_2}} 
	\ap \f{0.9879391307\ld}{0.00001220147952\ld} \ap 80968.7980038\ld
	\ap \f{\pi^2}{6} \f{(4\pi)^6}{5!} \b( 1+\f{1}{3} \e) \b(1+\f{1}{9}\e)\1 . \]
The example $f_3 = q - 4q^3 - 2q^5 + \cd$ of weight 4, level 8, and trivial
character has a supercuspidal local component at $p = 2$. By Proposition 2.5
of \cite{NelsonPitaleSaha2014} we know $N_{\ad,2} = 16$, $L_2(\ad f,1) = 1$,
and $(*)_2 = 1+\f12$. Here the computation gives 
\[ \f{L(\ad f_3,1)}{\g{f_3,f_3}} 
	\ap \f{0.8047560912\ld}{0.0000784759013\ld} \ap 10254.8180648\ld
	\ap \f{\pi^2}{6} \f{(4\pi)^4}{3!} \b( 1 + \f{1}{2} \e) . \]
For the newform $f_4 = q + 6\rt{10} q^2 + 232q^4 - 96\rt{10}q^5 + \cd$ of
weight 8, level 9, and trivial character we can find (either by a computation
via Loeffler-Weinstein's algorithm, or by trial and error with the $L$-function
parameters) find that $\pi_{f,3}$ is isomorphic to its twist by $\et$ and we
have $N_{\ad,3} = 9$, so $L_3(\ad f,1) = (1+1/3)\1$ and $(*)_3 = 1$ and sure
enough 
\[ \f{L(\ad f_4,1)}{\g{f_4,f_4}} 
	\ap \f{1.6698026860\ld}{8.2275074570\ld\dt 10^{-6}} \ap 202953.652096\ld
	\ap \f{\pi^2}{6} \f{(4\pi)^8}{7!} . \]


\subsection{Comments on computing minimal twists} 
\label{sec:ComputingMinimalTwists}

Thus far, all of our algorithms have appeared to achieve our goal of avoiding
ever working with full spaces of modular forms of a given weight, level, and
character. Instead, if we are given the $q$-expansion of a modular form $f$, we
have at worst needed to work with a collection of twists of it. However, there
is a bit of a caveat to this: to correctly find all of the twists of $f$ and
their levels, we need to start with a \E{minimal} twist of $f$. 
\par
In practice, for most cases we work with $f$ will either be twist-minimal in
the first place, or we will have specifically picked it out as a twist of a
lower-level form. But in general a minimal twist needs to be searched for. We
do this by a brute-force search of lower-level modular forms, and
Loeffler-Weinstein \cite{LoefflerWeinstein2012} have a more sophisticated
algorithm. Both of these approaches involve computing full spaces of modular
forms, however, and it would be desirable to have an algorithm that doesn't.
\par
One approach we could consider taking would be to start with $f$ of some level
$N$, take its naive twists $f_\mu$, and then check numerically if $f_\mu$ is
actually of some smaller level; since $f_\mu$ is automatically modular under
some $\Ga_0(N')$, to check modularity under any $\Ga_0(M)$ we'd just need to
check whether it transforms correctly under
\[ \M{cc}{ 1 & 0 \\ M & 1 }. \]
It would be straightforward to check it the transformation rule appears to hold
numerically for a handful of points. This would not provide a \E{proof} that
$f_\mu$ is modular of our lower level, but in the spirit of the numerical
computations in this paper it would be a strong justification. 
\par
The hole in this strategy is that it only checks modularity of the naive twist
$f_\mu$ but we know in some cases the true twist $f\ox\mu$ will have extra
Fourier coefficients at bad primes that were  ``twisted away'' in $f$. To deal
with all cases, we would need a way to recover the lost coefficients of $f$ at
bad primes, either theoretically or numerically. We are not sure if there is a
known way to do this, and in any case have not pursued it since the brute-force
approach is sufficient for the cases we want to handle.


\subsection{Computing a ratio of Petersson inner products}
\label{sec:PeterssonRatio}

One feature of our method for computing Fourier expansions, and thus Petersson
inner products, is that it doesn't require the modular forms involved to be
newforms. Even with the method described in \ref{sec:LeastSquaresEigenbasis},
we can take $f$ to be any oldform associated to a newform $f_0$ and work with
$f$ directly, only needing to use $f_0$ itself to determine a basis for the
space $f|[\al_h]_k$ lies in. This is useful for our purposes of numerically
verifying computations made in \cite{Collins2016}, as some of these
calculations involve taking a newform $h$ and relating $\g{h,h}$ to
$\g{h',h''}$ where $h',h''$ are particular oldforms associated to $h$. We give
a few examples of computations verifying such calculations here, illustrating a
simpler version of the more complex comparisons needed to be made in
\cite{Collins2016}.
\par
For instance, in Section 6.2 of \cite{Collins2016} we calculate the formula 
\[ \f{\g{h(pz),h(z)}}{\g{h(z),h(z)}} = \f{a_p}{p^{m-1}(p+1)} \]
when $h(z) = \s a_n q^n$ is a weight-$m$ eigenform the prime-to-$p$ Hecke
operator $T(p)$. We can then numerically check this in the case $h = \De$ is
the $\De$-function and $p = 11$ (so $a_p = 534612$), where we get 
\[ \f{\g{\De(11z),\De(z)}}{\g{\De(z),\De(z)}} 
	\ap \f{1.5438373630\ld\dt 10^{-13}}{9.8869793538\ld\dt 10^{-7}} 
	\ap 1.5614853715\ld\dt 10^{-7} \ap \f{534612}{11^{11}\dt 12}. \]
This formula was used as an intermediate in \cite{Collins2016} for computations
with $p$-stabilizations of a $p$-ordinary form $h$ (one where $a_p$ is not
divisible by $p$). If we let $\al_p$ and $\be_p$ be the roots of the Hecke
polynomial for $a_p$ such that $\al_p$ is a $p$-adic unit for a given embedding
$\L\Q \into \C$ and $\be_p$ is not, then one can define the $p$-stabilization
as $h^\sh(z) = h(z) - \be_p h(pz)$ and also $h^\ft = h(z) - \al_p h(pz)$. We
then calculated that $h^\na(z) = h(z) - p\be_p h(pz)$ is orthogonal to $h^\ft$
under the Petersson inner product, which allowed us to realize ``projection
onto $h^\sh$'' as a scalar multiple of the functional $\g{-,h^\na}$, and proved
the following formula
\[ \f{\g{h^\sh,h^\na}}{\g{h,h}} 
		= \f{(-\al/\be)(1 - \be/\al)(1 - p\1 \be/\al)}{(1+p\1)}. \]
This ratio of arises when determining removed Euler factors in the $p$-adic
$L$-functions we were working with.
\par
In our example of $h = \De$ and $p = 11$ (the smallest prime for which $\De$ is
$p$-ordinary), we take $\al,\be$ to be the roots $(a_{11}\pm\rt{a_{11} - 4\dt
11^{11}})/2$ and we can numerically compute that $\g{\De^\ft,\De^\na} \ap 0$
and moreover that 
\[ \f{\g{\De^\sh,\De^\na}}{\g{\De,\De}} 
	\ap \f{1.4821834825\ld\dt 10^{-6} + 7.1394620388\ld\dt 10^{-7} i}
			{9.8869793538\ld\dt 10^{-7}}
	\ap 1.4991267095\ld-0.7221075096\ld i \]
which does indeed agree with the expected ratio above. 
 
\section{Inner products involving three eigenforms} 
\label{sec:InnerProductsThreeEigenforms}
\subsection{Working with eigenforms of different levels}
\label{sec:EigenformsDifferentLevels}

In this section we give examples of computations involving Petersson inner
products of the form $\g{fg,h}$ where $f,g,h$ are three modular forms of levels
$k$, $m-k$, and $m$ for $0 < k < m$; thus the product $fg$ is of weight $m$ and
it makes sense to pair it with $h$. We will generally also assume they satisfy
$\ch_f \ch_g = \ch_h$ since otherwise the inner product is trivially zero. 
\par
Once again, our general setup will be that we are given the $q$-expansions of
$f$, $g$, and $h$ at infinity. Since the $q$-expansion of $fg$ is just the
product of the $q$-expansions of $f$ and $g$, we can apply Theorem
\ref{thm:PeterssonFormula}, which we can write out explicitly as follows. 

\begin{thm}
	Suppose $k,m$ are integers satisfying $0 < k < m$, and $f = \s_n a_n q^n \in
	S_k(N_f,\ch_f)$, $g = \s b_n q^n \in S_{m-k}(N_g,\ch_g)$, and $h = \s c_n
	q^n \in S_m(N_h,\ch_h)$ are three cusp forms. Set $N = \lcm(N_f,N_g,N_h)$.
	Then we have 
	\[ \g{fg,h} = \f{4}{\vol(\DH\q\Ga)} 
			\s_s \f{h_{s,0}}{h_s} \s_{n=1}^\oo 
				\b( \s_{i=1}^{n-1} a_{i,s} b_{n-i,s} \e) \f{\Lc_{n,s}}{n^{k-1}} 
			\s_{m=1}^\oo \pf{x}{8\pi}^{k-1} \7( x K_{k-1}(x) - K_{k-2}(x) \7) \]
	setting $x = 4\pi m\rt{\f{n}{h_s}}$. Again $K_v$ is a $K$-Bessel function,
	$s$ runs over all cusps of $\Ga_0(N)$, $h_{s,0}$ is the width of that cusp
	for $\Ga_0(N)$, and $h_s$ is a common width such that if we fix a matrix
	$\al_1$ taking $\oo$ to $s$ then $f|[\al_{h_s}]_k = \s a_{n,s} q^n$,
	$g|[\al_{h_s}]_{m-k} = \s b_{n,s} q^n$, and $h|[\al_{h_s}]_m = \s c_{n,s}
	q^n$ all have integer exponents of $q$.
\end{thm}

Thus we can numerically compute $\g{fg,h}$ by numerically computing the
$q$-expansions of $f$, $g$, and $h$ at each cusp of $\Ga_0(N)$ and applying
this formula just like in Algorithm \ref{algo:PeterssonForCuspFormsSameLevel}.
In practice we will want to implement this slightly differently, because
usually $N_f$, $N_g$, and $N_h$ will be distinct so we only need to compute
expansions for $f$ at the cusps of the congruence subgroup it's naturally
defined over. Doing this requires modifying our algorithm to first look over
all cusps of $\Ga_0(N)$ and note which is the most accuracy we need from each
expansion of $f,g,h$, then compute each of these expansions for the ``natural''
cusps, and finally use them to get the appropriate expansions for each cusp of
$\Ga_0(N)$ (remembering that expansions at different representatives of a
single cusp will differ as explained in Proposition
\ref{prop:ExpansionAtEquivalentCusps}). 
\par
A first application of the above methods is to verify the computations in
Section 6.4 of \cite{Collins2016}, where we compare a Petersson inner product
$\g{fg,h}$ compared with a $p$-stabilized version $\g{fg^\sh,h^\na}$; these
give rise to the removed Euler factors at $p$ for the $p$-adic $L$-functions we
construct. The setup is similar to what was described above in Section
\ref{sec:PeterssonRatio}; we do not go into detail beyond saying that we have
run a variety of numerical examples that agree with our computed formulas
(which also agree with the conjectured form for removed Euler factors in
general).


\subsection{Numerically verifying an explicit Ichino formula}
\label{sec:ExplicitIchino}

We now turn to our main application, of offering various numerical
verifications of an explicit form of Ichino's triple-product formula needed in 
\cite{Collins2016}. Ichino \cite{Ichino2008} proved a general result about
automorphic representations on $\GL_2$, which can be applied to the case of
three holomorphic newforms $f,g,h$ (of compatible weights and characters, as
discussed previously) to obtain a formula relating $|\g{fg,h}|^2$ to the
central value of the triple-product $L$-function $L(f\x g\x\Lh,s)$. It is clear
in principle that the formula will give us an explicit constant (a certain
power of $\pi$ times an algebraic number) relating these two quantities.
However, determining the algebraic part of the constant may involve many
delicate calculations, and our goal is to provide a computational verification
of the resulting formula.
\par
Specifically, in Theorem 3.1.2 of \cite{Collins2016} we establish the following
explicit version of Ichino's formula. We remark that if $f,g,h$ are newforms
such that one of them is new at a prime $p$ and the other two are old there,
then $\g{fg,h}$ is automatically zero; the factors $M_f,M_g,M_h$ are introduced
to avoid this.

\begin{thm}
	Fix integers $m > k > 0$, and let $f \in S_k(N_f,\ch_f)$, $g \in
	S_{m-k}(N_g,\ch_g)$, and $h \in S_m(N_h,\ch_h)$ be classical newforms such
	that the characters satisfy $\ch_f \ch_g = \ch_h$. Take $N_{fgh} =
	\lcm(N_f,N_g,N_h)$ and choose positive integers $M_f,M_g,M_h$ such that the
	three numbers $M_f N_f$, $M_g N_g$, $M_h N_h$ divide $N_{fgh}$ and moreover
	none of the three is divisible by a larger power of any prime $p$ than both
	of the others. Then we have 
	\[ |\g{f_{M_f}g_{M_g},h_{M_h}}|^2 
		= \f{3^2 (m-2)! (k-1)! (m-k-1)!}
					{\pi^{2m+2} 2^{4m-2} M_f^k M_g^{m-k} M_h^m} 
			L(f\x g\x\Lh,m-1) \p_{p|N_{fgh}} I_p^{**}, \]
	where $f_{M_f}(z)$ denotes $f(M_f z)$, and the constants $I_p^{**}$ are
	values of (slightly re-normalized) ``Ichino local integrals''. 
\end{thm}

The bulk of the difficulty in making this completely explicit is in determining
the constants $I_p^{**}$ at the bad primes. Before starting on what is known
about this we first want to check the formula for newforms of level 1 to verify
that the other part of the constant is correct (especially the power of $2$ in
the denominator). In the case where $f,g,h$ are all of level 1 the formula
reduces to 
\[ |\g{fg,h}|^2 = \f{3^2 (m-2)! (k-1)! (m-k-1)!}
					{\pi^{2m+2} 2^{4m-2}} L(f\x g\x\Lh,m-1). \]
The simplest case to test is when $f = g$ is the $\De$-function of weight 12
and $h$ is a newform of weight 24 (there are two conjugate such newforms, but
for explicitness pick the one with $540 - 12\rt{144169}$ as the coefficient of
$q^2$). We can compute $\g{fg,h}$ by our usual algorithm, and $L(f\x
g\x\Lh,m-1)$ via Dokchitser's algorithm \cite{Dokchitser2004}. Since all of our
forms are of level 1 the conductor of this $L$-function is $1$ and all of the
Euler factors are the naive triple-product ones. (The other parameters for
Dokchitser's algorithm follow from the analytic theory of such $L$-functions
and doesn't depend on the levels: the weight is $2m-2$, the local constant is
$1$, and the gamma factors are $0$, $1$, $-k+1$, $-k+2$, $-(m-k)+1$,
$-(m-k)+2$, $-m+2$, and $-m+3$). Running this we get:
\[ \f{|\g{fg,h}|^2}{L(f\x g\x\Lh,23)} 
	\ap \f{1.2769689139\ld\dt 10^{-16}}{1.1302460925\ld}
	\ap 1.1298149335\ld\dt 10^{-16} 
	\ap \f{3^2 \dt 22!\dt 11!\dt 11!}{\pi^{50} 2^{94}}. \]
With the main constant in the formula verified, we can move on to checking
local factors $I_p^{**}$ in various cases. This local factor arises as follows
(which we explain in detail in Section 3.2 of \cite{Collins2016}): first $I_p$
is defined as a local integral of matrix coefficients of newvector of local
constituents, then it is normalized by some $L$-factors to a value $I_p^*$
(which is the standard quantity considered in the literature), and we modify it
slightly further to get the constant $I_p^{**}$ appearing in our formula.
(Specifically, in the process of making Ichino's formula explicit we get
$\g{f,f}$ on one side and $L(\ad f,1)$ on the other, and similarly for the
other two forms, so $I_p^{**}$ takes into account the factors $(*)_p$ arising
form this comparison as detailed in Section
\ref{sec:ComparisonWithAdjointLValues}). 

\subsubsection*{Case of one conductor-$p$ special representation and two
unramified representations.} The simplest nontrivial case for our local
integrals is when $\pi_{f,p}$, $\pi_{g,p}$, and $\pi_{h,p}$ (the local
representations at $p$ for our three newforms $f,g,h$) consist of two
unramified representations and one special representation of conductor $p$, in
some order. In this case the local integral was calculated by Woodbury in
\cite{Woodbury2012} to give
\[ I_p^* = \f{1}{p} \b( 1 + \f{1}{p} \e)\1 
		\qq\qq I_p^{**} = \f{1}{p} \b( 1 + \f{1}{p} \e)^{-2}. \]
Also in this case, the local factor of the $L$-function is 
\[ L_p(f\x g\x\Lh,s) = \p_{i,j=1}^2 ( 1 - \al_i \be_j \ga p^{-s} )\1 \]
where $\al_1,\al_2$ and $\be_1,\be_2$ are the roots of the Hecke polynomials at
$p$ for the two of $f,g,\Lh$ that are unramified, and $\ga$ is the coefficient
of $p$ for the one that is special. The local contribution to the conductor of
the functional equation is $p^4$. 
\par
As a numerical verification, we apply Dokchitser's algorithm to compute $L(f\x
g\x\Lh,17)$ and ours to compute $|\g{f(z)g(3z),h(z)}|^2$ where $f$ is the
unique newform of weight 6 and level 3, $g$ is the unique newform of weight 12
and level 1 (the $\De$-function), and $h$ is the unique newform of weight 18
and level 1. Running this computation gives 
\[ \f{|\g{f(z)g(3z),h(z)}|^2}{L(f\x g\x\Lh,17)} 
	\ap \f{4.7335974505\ld\dt 10^{-23}}{1.3684877005\ld}
	\ap 3.4589988997\ld\dt 10^{-23} 
	\ap \f{3^2 \dt 16!\dt 5!\dt 11!}{\pi^{38} 2^{70} 3^{12}} 
			\dt \f{1}{3} \b(1+\f13\e)^{-2} . \]
We remark for this computation (and the ones to follow), the time-intensive
part is computing the $L$-value. For a computation that resulted in about 15
decimal points of accuracy in the case above, the $L$-function algorithms built
into Sage asked for over 30000 terms of the Dirichlet series, which in turn
required finding the coefficients of the three modular forms at all primes up
to at least 30000. Using the default modular symbol methods in Sage for working
with modular forms, this took several hours on the author's laptop computer - a
lengthy computation but not one requiring special resources. 

\subsubsection*{Case of two conductor-$p$ principal series representations and
one unramified representation.} The next case we can consider is when two of
our representations are principal series of conductor $p$. We carry out this
computation in \cite{Collins2016}, and obtain the following local factors:
\[ I_p^* = \f{1}{p} \qq\qq I_p^{**} = \f{1}{p} \b( 1 + \f{1}{p} \e)^{-2}. \]
Again the conductor is $p^4$, and the local $L$-factor is  
\[ L_p(f\x g\x\Lh,s) = \p_{i=1}^2 
		(1 - \al_i \be \ga p^{-s})\1 (1 - \al_i\1 \be\1 \ga\1 p^{-s})\1 \]
where as before $\al_1,\al_2$ are the roots of the Hecke polynomial for the
one of $f,g,\Lh$ unramified at $p$ and $\be,\ga$ are the $p$-th coefficients
for the other two. 
\par
As a test of this particular case, we take $f = g$ to both be the newform $q -
2i\rt{11}q^2 + 6i\rt{11}q^3 + \cd$ of weight 6, level 5, and of the unique even
character $\ch$ of conductor 5, and again take $h = \De$. This gives 
\[ \f{|\g{f(z)g(z),h(z)}|^2}{L(f\x g\x\Lh,11)} 
	\ap \f{1.6015746784\ld\dt 10^{-16}}{1.4547492648\ld} 
	\ap 1.1009283297\ld\dt 10^{-16}
	\ap \f{3^2 \dt 10!\dt 5!\dt 5!}{\pi^{26} 2^{46}} 
			\dt \f{1}{5} \b(1+\f15\e)^{-2}. \]

\subsubsection*{Other conductor-$p$ cases.} There are a handful of other cases
to consider where all representations are of conductor $\le p$, and most have
been computed in the literature. We do not need these in the specific setup
considered in \cite{Collins2016}, but we have carried out numerical
computations as a verification of each of them as well. 
\begin{itemize}
	\item Two conductor-$p$ special and one unramified: Here the local
		conductor is $p^4$ and $L_p(f\x g\x\Lh,s) = \p_{i=1}^2 (1 - \al_i \be
		\ga p^{-s})\1 (1 - \al_i \be \ga p^{-s+1})\1$ where once again
		$\al_1,\al_2$ are the roots of the Hecke polynomial for the one of
		$f,g,\Lh$ unramified at $p$ and $\be,\ga$ are the $p$-th coefficients for
		the other two. In this case the local factor was worked out by Woodbury
		\cite{Woodbury2012} as $I_p^* = \f{1}{p}$ and $I_p^{**} = \f{1}{p} (1
		+ \f{1}{p})^{-2}$. 
	\item Three conductor-$p$ special: The local conductor is $p^5$ and 
		$L_p(f\x g\x\Lh,s) = (1-\al\be\ga p^{-s})\1(1 - \al\be\ga
		p^{-s+1})^{-2}$ where $\al,\be,\ga$ are the $p$-th coefficients of
		$f,g,\Lh$. This is the only case where the $\ep$-factor for $L_p(f\x
		g\x\Lh,s)$ is not automatically one, and is instead given by
		$-\al\be\ga/p^{m-2}$; here the local factor is also calculated by
		Woodbury \cite{Woodbury2012} as $I_p^* = (1-\ep)\f1p(1+\f1p)$ and 
		$I_p^{**} = (1-\ep) \f{1}{p} (1 + \f{1}{p})^{-2}$. 
	\item Two conductor-$p$ principal series and one conductor-$p$ special: The
		local conductor is $p^6$ and $L_p(f\x g\x\Lh,s) = (1-\al\be\ga
		p^{-s})\1(1 - \al\1\be\1\ga\1 p^{-s})$ where $\al,\be,\ga$ are the $p$-th
		coefficients of $f,g,\Lh$. In this case the local factors are computed
		by Humphries \cite{Humphries2018} giving $I_p^* = \f1p(1+\f1p)$
		$I_p^{**} = \f{1}{p} (1 + \f{1}{p})^{-2}$. 
	\item Three conductor-$p$ principal series: The local conductor is $p^6$ and 
		$L_p(f\x g\x\Lh,s) = (1-\al\be\ga p^{-s})\1(1 - \al\1\be\1\ga\1
		p^{-s})^{-2}$ where $\al,\be,\ga$ are the $p$-th coefficients of
		$f,g,\Lh$. We do not know of a place in the literature explicitly dealing
		with this case, but it should be a special case of the results of Hsieh
		\cite{Hsieh2017}. A numerical test suggests the values should be $I_p^* =
		\f1p(1+\f1p)$ and $I_p^{**} = \f{1}{p} (1 + \f{1}{p})^{-2}$. 
\end{itemize}
Surprisingly, the Ichino local integrals seem much more uniform across these
various cases when expressed in our modified normalization $I_p^{**}$ (intended
for working with classical modular forms) than the standard one $I_p^*$ coming
from the adelic formulation. We remark that the factors of $(1+\f1p)^{-2}$
arise from us normalizing our Petersson inner products by $\Vol(\DH\q\Ga)$. If
we left it unnormalized instead then the local integrals would have an even
simpler form. 

\subsubsection*{Case of one representation of conductor $\ge p^2$ and two
unramified representations.} We can also consider the general case when two of
our three representations $\pi_1,\pi_2$ are unramified. The case when the third
representation is conductor-$p$ special was discussed above, so we're left with
the case of conductor $p^c$ for $c \ge 2$.
\par
The overall condition that the product of central character is trivial forces
$\pi_3$ to have an unramified central character itself. The most interesting
case of such a $\pi_3$ is when it is supercuspidal, but there's also the
possibility a principal series (corresponding to two characters with the
product unramified) or a special representation (a twist of the conductor-$p$
one by a character with its square unramified). In all three cases, the local
$L$-factor $L_p(f\x g\x\Lh,s)$ is trivial, but the local conductor for the
$L$-function is $p^{4c}$, making the algorithms for finding the $L$-value quite
computationally intensive. 
\par
Each of these three cases needs to be analyzed separately, all of them are
considered in \cite{Hu2017}, and using our normalization $I_p^{**}$ all of them
have the same form
\[ I_p^{**} = \f{1}{p^c} \b( 1 + \f{1}{p} \e)^{-2}. \]
As before, we can check an assortment of examples for these cases. For
instance, as a test of the supercuspidal case we can take $f = q - 12q^3 +
54q^5 + \cd$ to be the unique newform of weight 6, level 4, and trivial
character, $g = \De$ the delta-function, and $h = q - 528q^2 - 4284 q^3 + \cd$
the unique newform of weight 18 and level 1. We then find 
\[ \f{|\g{f(z)g(4z),h(z)}|^2}{L(f\x g\x\Lh,17)} 
	\ap \f{4.2746854\ld\dt 10^{-25}}{0.6583584\ld} 
	\ap 6.4929462\ld\dt 10^{-25}
	\ap \f{3^2 \dt 16!\dt 5!\dt 11!}{\pi^{38} 2^{70} 4^{12}} 
			\dt \f{1}{4} \b(1+\f12\e)^{-2} . \]
We have carried out similar computations checking the special and principal
series cases as well. 

\subsubsection*{One case with two supercuspidals and one unramified
representation.} Finally, the last case we will consider is where one of the
local representations $\pi_{f,p}, \pi_{g,p}, \pi_{h,p}$ is unramified, and the
other two representations are both isomorphic to a single supercuspidal $\pi$
with trivial central character. In particular we'll consider the case where
$\pi \iso \pi\ox\et$ where $\et$ is the unramified quadratic character of
$\Q_p^\x$; this is ``type 1'' in the notation of \cite{NelsonPitaleSaha2014}.
Nelson-Pitale-Saha prove that in this case we have 
\[ I_p^* = I_p^{**} = \f{1}{p^c} \b( 1 + \f{1}{p} \e)^{-2}
	\pf{(\al^{c/2+1}-\al^{-c/2-1}) 
			- p\1 (\al^{c/2-1}-\al^{-c/2+1})}{\al-\al\1}^2, \]
where $p^c$ is the conductor of $\pi$, we assume $p^c$ is also the conductor
of $\pi\x\pi$ (which will be true in our cases of interest), and $\al,\al\1$
are the Satake parameters of the third unramified representation (so, for
example, if $f$ is the one unramified at $p$, then $\al p^{(k-1)/2}, \al
p^{(k-1)/2}$ are the roots of the Hecke polynomial for $f$ at $p$). 
\par
In this situation, the local conductor is $p^{2c}$ and the local $L$-factor
\[ L(f\x g\x\Lh,s) = (1 - \al p^{-s})\1 (1 + \al p^{-s})\1 
		(1 - \al\1 p^{-s})\1 (1 + \al\1 p^{-s})\1. \]
We can then proceed to numerical examples. Our first example will take $f = g =
q + 6\rt{10} q^2 + 232 q^4 + \cd$ to be the newform of weight 8, level 9,
trivial character, and which isn't a twist of a newform of level 3, and $h = q
+ 216 q^2 - 3348 q^3 + \cd$ is the unique newform of weight 1 and level 16. We
then numerically compute 
\begin{multline*}
\f{|\g{f(z)g(z),h(z)}|^2}{L(f\x g\x\Lh,15)} 
	\ap \f{2.1021427352\ld\dt 10^{-17}}{7.9702799221\ld} \\
	\ap 2.6374766705\ld\dt 10^{-18} 
	\ap \f{3^2 \dt 14!\dt 7!\dt 7!}{\pi^{34} 2^{62}} 
			\dt \f{1}{3^2} \b(1+\f13\e)^{-2} \dt \pf{3348^2}{3^{15}}. 
\end{multline*}
Here, $c = 2$ so the last term in $I_p^{**}$ becomes
$(\f{\al^2-\al^{-2}}{\al-\al\1})^2 = (\al+\al\1)^2$ (note the second half of
the numerator drops out since $c/2-1 = 0$), and by definition $\al+\al\1$ is
just $a_p/p^{(m-1)/2}$. 
\par
For our second example we take $f = g = q - \f{3+\rt{129}}{2} q^2 +
\f{5+3\rt{129}}{2} q^4 + \cd$ to be a newform of weight $6$, level $81$, and
trivial character (the unique such newform that isn't a twist of a lower-level
form), and $h = \De$ the delta-function. We then compute
\begin{multline*}
\f{|\g{f(z)g(z),h(z)}|^2}{L(f\x g\x\Lh,11)} 
	\ap \f{4.2156534297\ld\dt 10^{-18}}{0.8058589132\ld} \\
	\ap 5.2312549509\ld\dt 10^{-18}
	\ap \f{3^2 \dt 10!\dt 5!\dt 5!}{\pi^{26} 2^{46}} 
			\dt \f{1}{3^4} \b(1+\f13\e)^{-2} 
			\dt \b( \f{252^2}{3^{11}} - 1 - \f13 \e)^2 . 
\end{multline*}
Here $c = 4$ so the final term in $I_p^{**}$ is 
$(\f{(\al^3-\al^{-3}) - p\1(\al-\al\1)}{\al-\al\1})^2$, giving $\al^2 + 1 +
\al^{-2} - p\1 = (\al+\al\1)^2 - 1 - p\1$. 

\subsubsection*{Other cases.} The results of Nelson-Pitale-Saha,
\cite{NelsonPitaleSaha2014} and Hu \cite{Hu2017} compute Ichino local integrals
$I_p^*$ in more generality than we have discussed, and more recently Hsieh
\cite{Hsieh2017} has computed them in many more situations. We do not claim to
have checked any of these beyond what is discussed above, but in principle this
could be done by the same sorts of calculations that we have given.

\bibliographystyle{amsalpha}
\bibliography{NumericalCusps_bib.bib}

\end{document}